\documentclass[final]{article}
\usepackage{amsmath, amssymb, amsthm}
\usepackage{subfig, float}
\usepackage{graphicx}
\usepackage{multirow, makecell}

\newtheorem{theorem}{Theorem}
\newtheorem{proposition}{Proposition}

\usepackage{color, soul}

\newcommand{\rmbf}[1]{\mathrm{\textbf{#1}}}
\newcommand{\midand}{\ \mathrm{and}\ }
\newcommand{\infinity}{\infty}
\newcommand{\middlebar}{\enskip \middle| \enskip}
\newcommand{\eqdot}{\ .}
\newcommand{\eqcom}{\ ,}

\newcommand{\N}{\mathbb{N}}
\newcommand{\Z}{\mathbb{Z}}
\newcommand{\R}{\mathbb{R}}
\newcommand{\lp}[1]{L^{#1}}
\newcommand{\hp}[1]{H^{#1}}
\newcommand{\hinv}{H^{-1}}
\newcommand{\ellnu}{\ell^1_\nu}

\newcommand{\norm}[1]{||#1||}
\newcommand{\normp}[2]{||#1||_{\lp{#2}}}
\newcommand{\normhp}[2]{||#1||_{\hp{#2}}}
\newcommand{\normX}[1]{||#1||_X}
\newcommand{\normnu}[1]{||#1||_{1,\nu}}
\newcommand{\normb}[1]{||#1||_{B(\ellnu)}}
\newcommand{\normbX}[1]{||#1||_{B(X)}}

\newcommand{\lapl}{\nabla^2}

\def\Xint#1{\mathchoice
{\XXint\displaystyle\textstyle{#1}}%
{\XXint\textstyle\scriptstyle{#1}}%
{\XXint\scriptstyle\scriptscriptstyle{#1}}%
{\XXint\scriptscriptstyle\scriptscriptstyle{#1}}%
\!\int}
\def\XXint#1#2#3{{\setbox0=\hbox{$#1{#2#3}{\int}$ }
\vcenter{\hbox{$#2#3$ }}\kern-.59\wd0}}

\def\dashint{\Xint-}

\newcommand{\pbar}{\bar{\psi}}
\newcommand{\abar}{\bar{a}}
\newcommand{\atil}{\widetilde{a}}

\begin{document}
\title{Microscopic Patterns in the 2D Phase-Field-Crystal Model}

\author{
  Gabriel Martine-La Boissoni\`ere\footnote{gabriel.martine-laboissoniere@mail.mcgill.ca}\\
  \and
  Rustum Choksi\footnote{rustum.choksi@mcgill.ca}
  \and
Jean-Philippe  Lessard\footnote{jp.lessard@mcgill.ca}
}

\date{%
    Department of Mathematics and Statistics, McGill University, Montr\'{e}al, QC, Canada
}

\maketitle
\begin{abstract}
Using the recently developed theory of rigorously validated numerics, we address the Phase-Field-Crystal (PFC) model at the microscopic (atomistic) level.  We show the existence of critical points and local minimizers  associated with ``classical'' candidates, grain boundaries, and localized patterns.  We further address the dynamical relationships between the observed patterns for fixed parameters and across parameter space, then formulate several conjectures on the dynamical connections (or orbits) between steady states.
\end{abstract}

\section{Introduction}
The Phase-Field-Crystal (PFC) model introduced in~\cite{ELDER_Elasticity} is a gradient system capable of modeling a variety of solid-state phenomena.  In its simplest form,  the PFC energy can be written as 
\begin{equation*}
E[\psi] = \dashint_\Omega \frac{1}{2} \left( \lapl \psi + \psi \right)^2 + \frac{1}{4} \left( \psi^2-\beta \right)^2
\end{equation*}
defined on phase-fields $\psi \in \hp{2}(\Omega)$ satisfying the \textit{phase constraint}
\begin{equation*}
\pbar = \dashint_\Omega \psi = \frac{1}{|\Omega|} \int_\Omega \psi \eqdot
\end{equation*}
The parameter $\beta$ represents inverse temperature such that $\beta = 0$ models maximum disorder. 
Coupled with this energy is its conservative $\hinv$ gradient flow  which entails the sixth-order PFC equation
\begin{equation*}
\psi_t = \lapl \left( \left( \lapl + 1 \right)^2 \psi + \psi^3 - \beta \psi \right).
\end{equation*} Note that the PFC model shares its energy with the Swift-Hohenberg equation~\cite{SWIFT_Hohenberg}, which is simply the $\lp{2}$ gradient flow of $E$. 
From linear stability analysis applied to single Fourier mode Ansatz, we find three main candidate global minimizers that divide parameter space, see the appendices. In the hexagonal lattice regime, 2D-simulations of the PDE starting with random noise quickly produce atoms that arrange into small patches of hexagonal lattices with random orientations. These patches grow 
and interact with each other, forming {\it grains} of hexagonal lattices  of atoms with a particular orientation. The morphology and evolution of these grains 
have  features resembling those in polycrystalline materials (cf.  Figure \ref{fig:grains}). 
\begin{figure}
	\centering
	\includegraphics[width=0.3\textwidth]{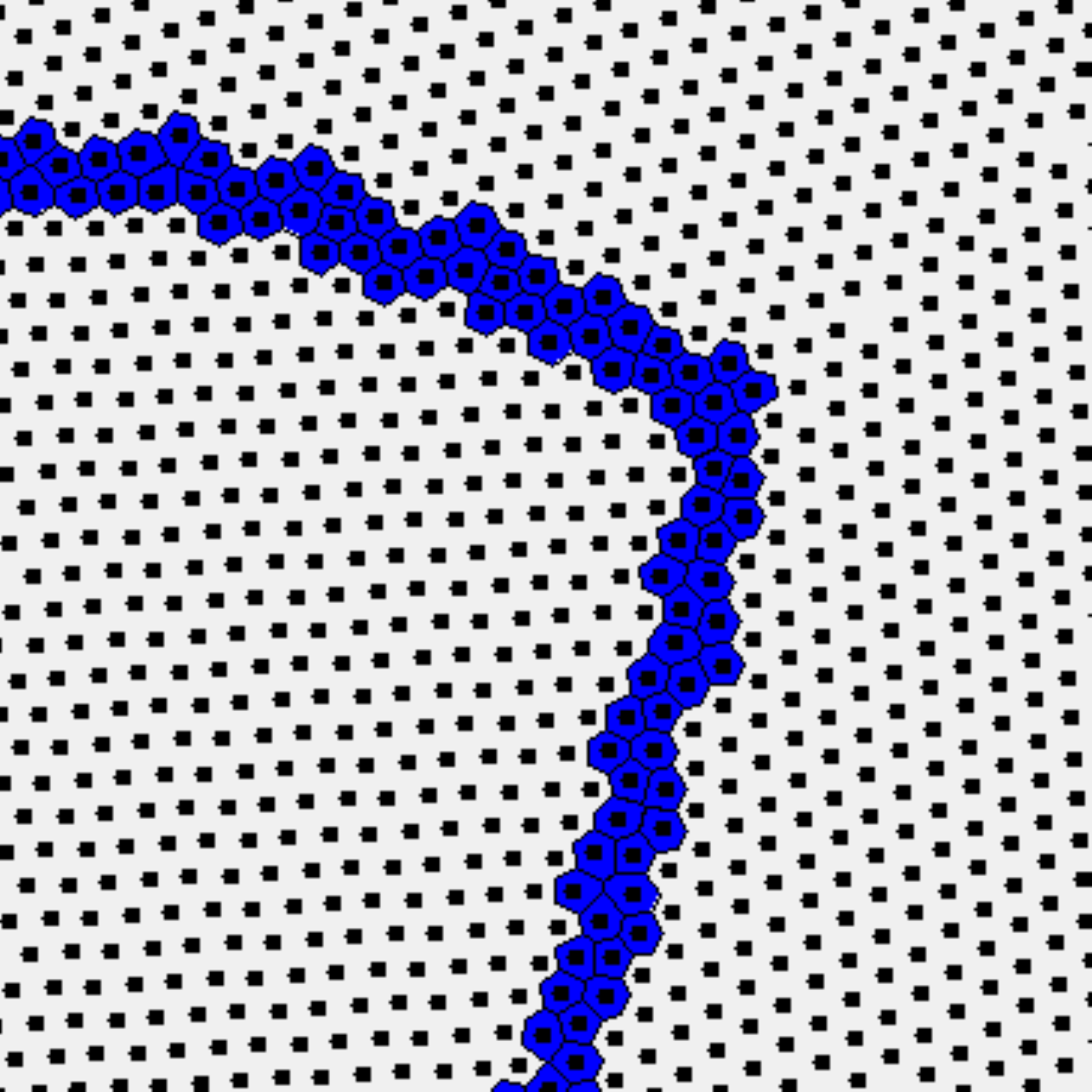}\qquad\qquad \qquad\qquad 
	\includegraphics[width=0.3\textwidth]{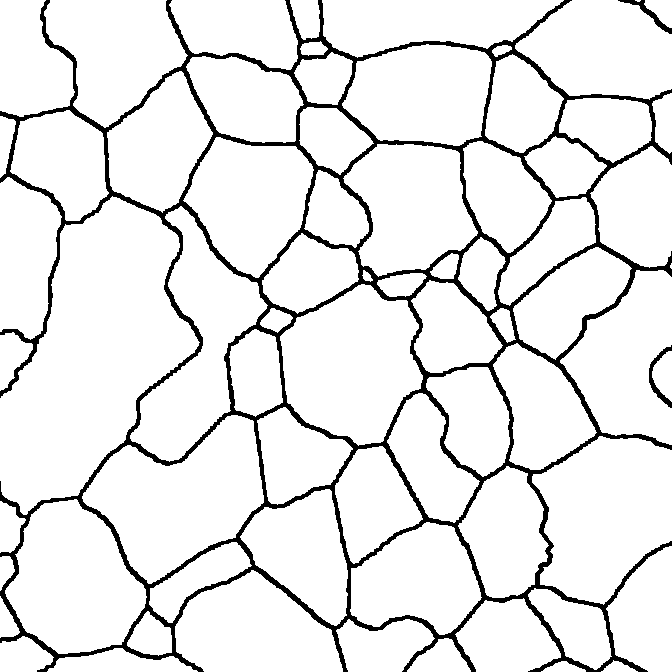}
	\caption{Left: Details of a grain boundary appearing in a PFC simulation (taken from \cite{MARTINE_AtomBased}). Right: 
Grain boundary network from a PFC simulation (taken from \cite{MARTINE_PFCStatistics}). Within each grain is a  hexagonal lattice of atoms with a particular orientation.}
	\label{fig:grains}
\end{figure}
In particular, it has recently been shown that statistics of many of experimentally observed (universal) grain boundary distributions are accurately captured by data amassed from simulations of this simple PFC equation
 \cite{BACKOFEN_GSD, MARTINE_PFCStatistics}.  While here we will mostly work  with this vanilla PFC formulation, we note that 
 a family of PFC-like equations can be derived from Density-Functional-Theory~\cite{EMMERICH_PFCReview} to obtain more complicated models capable of simulating eutectic and dendritic solidification~\cite{GREENWOOD_Eutectic} and graphene structures~\cite{SEYMOUR_Graphene, HIRVONEN_Graphene}.

In this article, we address  the PFC model and its steady states at the ``microscopic" level - the local atomic arrangement.  We believe that such an investigation of microscopic pattern-formation capabilities of PFC is not only of mathematical interest but is also necessary to construct ``designer'' models for polycrystalline behaviour.
For example, varying the parameters in the energy lead to more complicated states than simple lamellar and hexagonal. These include 
 localized patterns in the ``glassy regime" - the transition at the liquid (constant) and solid (hexagonal) transitions  -  and ``globules'' at large $\beta$.

With the exception of the constant (liquid) state (cf. \cite{SHIROKOFF_GlobalMinimality}), it is difficult to prove any theorem on the exact nature of steady states, local and global minimizers to this diffuse interface problem. 
What exists in the physics literature is numerical simulations, standard linear stability analysis, and Ansatz-driven energy comparisons. 
The recently developed theory of rigorously validated numerics (cf. ~\cite{KOCH_ComputeAssisted, NAKAO_VerifiedPDE, TUCKER_ValidatedIntroduction, VANDENBERG_Dynamics, GOMEZ_PDESurvey}) now provides a powerful new tool to 
bridge what can be observed numerically with rigorous statements on pattern morphology.  In a nutshell this approach can be summarized as follows: 
Given an approximate steady state, we use the {\it Contraction Mapping Theorem} to imply the existence and local uniqueness of an \textit{exact} steady state within a \textit{controlled distance} of the approximation. 
This notion of closeness is strong enough to imply further useful results, including closeness in energy and stability results. In this paper we use this new approach to address the following aspects of the PFC model:
\begin{itemize}
\item  Are the ``classical'' candidates obtained from linear stability analysis  close to actual local minimizers? 
\item Are the stable yet complicated patterns observed numerically indeed critical points in the PFC energy landscape? For example, are grain boundaries steady states or simply metastable states? 
\item What are the \textit{dynamical} relationships between the observed patterns for fixed parameters and across parameter space?
\end{itemize}
Based upon our results we formulate several conjectures on the connections (or orbits) between steady states. 
Taken as a whole, our work presents the first step into a rigorous analysis of the rich PFC energy landscape.

The outline of this  paper is as follows. We first setup the PFC equation in Fourier space and discuss the application of the framework of rigorous computations. We then verify the existence of 
important steady states of the PFC equation, including localized patterns and grain boundaries. With these states in hand,  we address the {\it energy landscape} of PFC with a discussion on conjectures for 
connections (or connecting orbits) between steady states. Finally, we presents results in one-parameter numerical continuation to outline some interesting features of the bifurcation diagram of PFC.

\section{PFC steady states in Fourier space}
We begin by writing the equation $\psi_t = 0$ in Fourier space to obtain a coupled system of equations for the Fourier coefficients of steady states. We will be slightly more general and consider functionals of the form 
\begin{equation*}
E[\psi] = \dashint_\Omega \frac{1}{2} (K \psi)^2 + \frac{1}{4}(\psi^2-\beta)^2
\end{equation*}
where $K$ is a linear differential operator $K$ acting on elements of a suitable function space. In particular,
\begin{equation*}
K = \begin{cases}
	\hfil \lapl + 1 \quad& \text{for the basic ``one-mode'' PFC model}\\
	(\lapl + 1) (\lapl + q^2) \quad& \text{for the ``two-mode'' PFC model~\cite{WU_TwoMode}}
\end{cases}
\end{equation*}
where $q$ is the secondary wavelength of two-mode PFC. Taking the $H^{-1}$ gradient flow of $E$, we obtain the PFC-like equation $\psi_t = \lapl \left( \left(K^2 - \beta \right) \psi + \psi^3 \right)$.

For simplicity, we let $\Omega$ be the rectangular domain $[0, L_x] \times [0, L_y]$ with periodic boundary conditions. We let
\begin{equation*}
L_x = \frac{4\pi}{\sqrt{3}} N_x \quad,\quad L_y = 4\pi N_y
\end{equation*}
where $N_x, N_y \in \N$ are the number of atoms lined up in the $x, y$-axes. The main parameters of the problem are then $(\pbar$, $\beta)$ and the domain size is given by $(N_x, N_y)$.

Let $a_\alpha$ be the Fourier coefficients of $\psi$ and let $(a_\alpha)_t$ be the time derivative. Inserting this expansion into the PFC equation results in an infinite system of equations of the form $(a_\alpha)_t = F_\alpha(a)$ thanks to orthogonality. The steady states may then be found numerically by solving $F(a) = 0$ up to some truncation order $M$. We will see later that it is imperative to isolate the zeros of $F$; the \textit{continuous} translational and rotational symmetries of PFC must then be broken. The simplest way to do so in this context is to also enforce Neumann boundary conditions. It is convenient to write $a_\alpha = a_{\alpha_1,\alpha_2}$ so that the symmetry and reality conditions become $a_{|\alpha_1|,|\alpha_2|} \in \R$.

This choice allows us to simplify a complex Fourier series into the cosine expansion
\begin{equation*}
\begin{split}
\psi(x, y) &= \sum_{\alpha \in \Z^2} a_\alpha \exp{\left( 2\pi i \frac{\alpha_1 x}{L_x} \right)} \exp{\left( 2\pi i \frac{\alpha_2 y}{L_y} \right)} \\
	&= \sum_{\alpha \in \N^2} W_\alpha a_\alpha \cos \left( \frac{2\pi \alpha_1}{L_x} x \right) \cos \left( \frac{2\pi \alpha_2}{L_y} y \right)
\end{split}
\end{equation*}
where $W$ is a \textit{weight matrix} defined by
\begin{equation*}
W_\alpha = \begin{cases}
	1 \quad& \text{if } \alpha = (0,0) \\
	2 \quad& \text{if } \alpha_1 = 0, \alpha_2 \neq 0 \text{ or } \alpha_1 \neq 0, \alpha_2 = 0 \\
	4 \quad& \text{otherwise} \eqdot
\end{cases}
\end{equation*}

The Fourier coefficients of $\lapl \psi$ are given by the elementwise product $L_\alpha a_\alpha$ where
\begin{equation*}
L_\alpha = -\left( \left( \frac{2\pi \alpha_1}{L_x} \right)^2 + \left( \frac{2\pi \alpha_2}{L_y} \right)^2 \right)
\end{equation*}
is the Fourier representation of the Laplacian. Inserting these expressions into the PFC equation and equating Fourier modes, we obtain
\begin{equation*}
(a_\alpha)_t = F_\alpha(a) = L_\alpha \left( \gamma_\alpha a_\alpha + (a * a * a)_\alpha \right)
\end{equation*}
where $*$ denotes the discrete convolution and the linear terms combining $K$ and $\beta$ are
\begin{equation*}
\gamma_\alpha = \begin{cases}
	\hfil \left( L_\alpha + 1 \right)^2 - \beta \quad& \text{for PFC}\\
	\left( L_\alpha + 1 \right)^2 \left( L_\alpha + q^2 \right)^2 - \beta \quad& \text{for two-mode PFC}  \eqdot
\end{cases}
\end{equation*}

Note that the $(0,0)$ Fourier component picks out the average phase so it is fixed to $\pbar$: this is consistent with $(a_{0,0})_t = 0$ thanks to $L_{0,0} = 0$. To keep track of the phase constraint directly in $F$, we replace its first trivial component by $F_{0,0} = a_{0,0} - \pbar$, resulting in:
\begin{equation*}
F_\alpha(a) = \begin{cases}
	a_{0,0} - \pbar \quad& \text{if } \alpha = (0,0) \\
	L_\alpha \left( \gamma_\alpha a_\alpha + (a * a * a)_\alpha \right) \quad& \text{otherwise} \eqdot
\end{cases}
\end{equation*}

The operator $F$ then represents the PFC dynamics in the sense that its zeros correspond to steady states of the PFC equation. A numerical advantage of the reduced expansion is that we effectively only have to compute a quarter of the full Fourier series. Obviously, this means we are not treating PFC in full generality over $\hp{2}$ and will have to address this later. As an aside, the equivalent $F$ for Swift-Hohenberg is simply $-(\gamma_\alpha a_\alpha + (a * a * a)_\alpha)$ hence its $(0,0)$ entry is nonzero and average phase is not conserved.

\section{Overview of rigorously validated numerics}
We present a brief overview of the recent framework of rigorously validated numerics for dynamical systems, see sources including~\cite{KOCH_ComputeAssisted, NAKAO_VerifiedPDE, TUCKER_ValidatedIntroduction, VANDENBERG_Dynamics} and~\cite{GOMEZ_PDESurvey} for a survey of techniques for PDEs.

Consider the Newton-like operator $T(a) = a - A F(a)$, where $A$ is a suitable inverse to the derivative $DF(a)$. On the one hand, if $T$ is a contraction on a closed ball, the contraction mapping theorem gives the existence and uniqueness of a zero of $F$ within this ball. On the other hand, the repeated application of $T$ (allowing $A$ to vary with $a$) should converge to this fixed point. We can then \textit{numerically} compute an approximate steady state $\abar$ for which $F(\abar) \approx 0$ up to numerical precision. If in addition we are able to show that $T$ is a contraction around $\abar$, then we immediately have the existence of an exact steady state $\atil$ close to $\abar$ in an appropriate metric. This relationship is made clear by the \textit{radii polynomial theorem}, so-called for reasons that will become clear shortly. To illustrate the method, we specialize the theorem to the case applicable to PFC, but see~\cite{DAY_ValidatedContinuation, HUNGRIA_RadiiPolynomial, BALASZ_RadiallySymmetric, VANDENBERG_RigorousChaos} for different approaches and~\cite{VANDENBERG_OhtaKawasaki, VANDENBERG_OhtaKawasaki2} for an application to 
Ohta-Kawasaki functional in 2D and 3D, respectively. Given Banach spaces $X, Y$, we use the notation $B(X, Y)$ for the space of bounded linear operators from $X$ to $Y$, $B(X) = B(X, X)$ and $B_r(a) \subset X$ for the open ball of radius $r$ around $a \in X$.

\begin{theorem}
Consider Banach spaces $X, Y$, a point $\abar \in X$ and let $A^\dagger \in B(X, Y)$, $A \in B(Y, X)$. Suppose $F : X \to Y$ is Fr\'echet differentiable on $X$ and $A$ is injective. In addition, suppose
\begin{equation*}
\begin{split}
\normX{AF(\abar)} &\leq Y_0 \\
\normbX{I - AA^\dagger} &\leq Z_0 \\
\normbX{A(DF(\abar) - A^\dagger)} &\leq Z_1 \\
\normbX{A(DF(b) - DF(\abar))} &\leq Z_2(r) r \enskip \forall b \in \overline{B_r(\abar)}
\end{split}
\end{equation*}
where $Y_0, Z_0, Z_1$ are positive constants and $Z_2$ is a positive polynomial in $r > 0$. Construct the \textit{radii polynomial}
\begin{equation}
\label{eq:RadiiPolynomial}
p(r) = Z_2(r) r^2 - (1 - Z_0 - Z_1) r + Y_0 \eqdot
\end{equation}

If $p(r_0) < 0$ for some $r_0 > 0$, then there exists a unique $\atil \in B_{r_0}(\abar)$ for which $F(\atil) = 0$.
\end{theorem}

The proof of this formulation is given in appendix~\ref{sm:RadiiPolyProof}, where we show a correspondence between the sign of the radii polynomial and the contraction constant of $T$: if $r_0$ can be found, $T$ is a contraction and the Newton iteration starting at $\abar$ must converge to some $\atil$. This proves not only the existence of the exact steady states but also gives control on its location in $X$ with respect to a known point. In practice, one finds an interval $[r_*, r^*]$ of radii for which $p(r)$ is negative; $r_* > 0$ gives the \textit{maximum} distance between $\abar \midand \atil$ while $r^* > r_*$ gives the \textit{minimum} distance between $\abar$ and \textit{another} zero of $F$. The zeros of $F$ must therefore be isolated for consistency.

Each bound may be understood intuitively: $Y_0$ being small indicates that $\abar$ is a good approximation of $\atil$ while $Z_1$ being small indicates that $A^\dagger$ is a good approximation for $DF(\abar)$, and so on. These bounds may be simplified analytically but must necessarily be computed numerically. Therefore, we ensure that our numerical computations go in the same direction as the required inequalities by using interval arithmetic~\cite{MOORE_IntervalAnalysis}, a formalized approach to deal with numerical errors. We used the interval arithmetic package INTLAB for MATLAB, see~\cite{RUMP_INTLAB, HARGREAVES_IntervalMATLAB}, to ensure that the radii polynomial approach is numerically rigorous.

This approach allows us to prepare numerical tools that can both find candidate steady states and compute the radii $r_*, r^*$ if they exist. If so, we immediately have a \textit{proof} that this candidate provides a good handle on an actual steady state of the PFC equation.

\section{Radii polynomial approach for PFC}
\label{sec:NewtonOperatorPFC}
Let us now apply these ideas to PFC by first computing $DF$ and the Newton operator. Let $\sigma$ represent the differentiation indices applied to $F_\alpha$. The derivative of $F_{0,0}$ is $1$ if $\sigma = (0,0)$ and $0$ otherwise, so we use the Kronecker delta notation to write
\begin{equation*}
\partial_{a_\sigma} F_{0,0} = \delta_{\sigma_1} \delta_{\sigma_2} \eqdot
\end{equation*}

For other values of $\alpha$, the linear terms similarly give
\begin{equation*}
\partial_{a_\sigma} \left( L_\alpha \gamma_\alpha a_\alpha \right) = L_\alpha \gamma_\alpha \delta_{\sigma_1-\alpha_1} \delta_{\sigma_2-\alpha_2} \eqdot
\end{equation*}

The derivative of the nonlinear triple convolution can be computed by differentiating with respect to all four $a_\alpha$ identified by symmetry. This algebraic computation is somewhat tedious but the result can be written succinctly as
\begin{equation*}
\begin{split}
\partial_{a_\sigma} (a*a*a)_\alpha = \frac{3 W_\sigma}{4} \Big( &(a*a)_{|\alpha_1+\sigma_1|, |\alpha_2+\sigma_2|} + (a*a)_{|\alpha_1+\sigma_1|, |\alpha_2-\sigma_2|}\\
	&\enskip + (a*a)_{|\alpha_1-\sigma_1|, |\alpha_2+\sigma_2|} + (a*a)_{|\alpha_1-\sigma_1|, |\alpha_2-\sigma_2|} \Big)
\end{split}
\end{equation*}
so that the full derivative of $F$ is:
\begin{align*}
\left[DF \right]_{\sigma, \alpha} (a) &= (\partial_{a_\sigma} F_\sigma)(a) \\
& = \begin{cases}
\delta_{\sigma_1} \delta_{\sigma_2} \quad& \text{if } \alpha = (0,0) \\
L_\alpha \left( \gamma_\alpha \delta_{\sigma_1-\alpha_1} \delta_{\sigma_2-\alpha_2} + \partial_{a_\sigma} (a*a*a)_\alpha \right) \quad& \text{otherwise}  \eqdot
\end{cases}
\end{align*}

$a, F$ and the convolutions may be viewed as infinite matrices whose ``top-left'' entry is the $(0,0)$ coefficient while the derivative $DF$ is an infinite 4-tensor. To implement the Newton method numerically, such objects must be truncated to order $M$ such that $a_\sigma = 0$ whenever \textit{either} $\sigma_1$ or $\sigma_2$ is greater than $M$. This results in the $(M+1)^2$ matrices $a^{(M)}, F^{(M)}$ while the derivative becomes the $(M+1)^4$ 4-tensor $DF^{(M)}$. Note that the $k$-convolution of $a^{(M)}$ has support $kM$ by definition.

We now introduce the Banach space framework. Let $\nu > 1$ and define $\ellnu(\Z^2)$ as the space of sequences $a_\alpha$ with finite norm
\begin{equation*}
\normnu{a} = \sum_{\alpha \in \Z^2} |a_\alpha| \nu^{|\alpha|} = \sum_{\alpha \in \Z^2} |a_\alpha| \nu^{|\alpha_1|+|\alpha_2|} \eqdot
\end{equation*}

The restriction of $\ellnu(\Z^2)$ using the symmetry condition is
\begin{equation*}
X = \left\{ a \in \ellnu(\Z^2) \middlebar a_\alpha = a_{|\alpha_1|, |\alpha_2|} \right\}
\end{equation*}
over which the norm simplifies to
\begin{equation*}
\normnu{a} = \sum_{\alpha \in \N^2} W_\alpha |a_\alpha| \nu^{|\alpha|} = \sum_{\alpha \in \N^2} |a_\alpha| \nu_\alpha
\end{equation*}
where $\nu_\alpha$ is a weight matrix that forces the fast exponential decay of the Fourier coefficients. The space $(X, \normnu{\cdot})$ can easily be shown to be Banach and the 2D discrete convolution forms a Banach algebra over it, immediate results from the triangle inequality and the fact that $\nu > 1$.

Let now $\abar, \atil \in X$ have the same meaning as before, with $\abar = 0$ outside of $U = \{0, 1, ..., M\}^2$ thanks to the truncation. Let $G = DF(\abar)^{(M)}$ and denote by $A^{(M)}$ the numerical inverse of $G$. We define \textit{approximate} operators $A^\dagger, A$ as
\begin{equation*}
A_{\alpha, \sigma}^\dagger = \begin{cases}
	G_{\alpha, \sigma} \quad& \text{if } \alpha, \sigma \in U \\
	L_\alpha \gamma_\alpha \quad& \text{if } \alpha = \sigma, \alpha \in \N^2 \backslash U \\
	0 \quad& \text{otherwise,}
\end{cases}
\quad
A_{\alpha, \sigma} = \begin{cases}
	A^{(M)}_{\alpha, \sigma} \quad& \text{if } \alpha, \sigma \in U \\
	L^{-1}_\alpha \gamma^{-1}_\alpha \quad& \text{if } \alpha = \sigma, \alpha \in \N^2 \backslash U \\
	0 \quad& \text{otherwise}
\end{cases}
\end{equation*}
which can be thought of as block tensors containing $G$ or its inverse paired with the \textit{linear terms} $L_\alpha \gamma_\alpha$ as the \textit{main ``diagonal''} of the second block. If $G$ is an invertible matrix,\footnote{The numerical method will fail if $G$ is almost singular, so this is the case in practice.} so is $A$ and it is thus injective. The inverse of $A$ is \textit{not} $A^\dagger$ however because $A^{(M)} G \approx I^{(M)}$ only up to numerical inversion errors.

Note that $F, DF \midand A^\dagger$ map to a space $Y$ with less regularity than $X$ because of the unbounded $L_\alpha \gamma_\alpha$ terms arising from \textit{real space derivatives}; $Y$ is a space where sequences $L_\alpha \gamma_\alpha a_\alpha$ have finite norm. However, the operator products against $A$ \textit{are} bounded on $X$ thanks to the \textit{fast} decay of $L_\alpha^{-1} \gamma_\alpha^{-1}$. Thus, we say that $A$ ``lifts'' the regularity of the other operators back to $X$, allowing statements such as $T : X \to X$ or $ADF(\abar) \in B(X)$.

We show in appendix~\ref{sm:SimplifiedBounds} how to simplify the bounds into expressions that can be evaluated numerically. This allows us to write down the radii polynomial $p(r) = Z_2(r) r^2 - (1 - Z_0 - Z_1) r + Y_0$, noting that $Z_2(r) = Z_2^{(0)} + Z_2^{(1)} r$, hence the polynomial is cubic with non-negative coefficients except for maybe the linear term. We have $p(0) > 0$, $p'(0) = Z_0+Z_1-1$ and $p(r) \to \infinity$ for large $r$. As a consequence, if $p$ is strictly negative for some positive $r$, there must exist exactly two strictly positive roots $r_* < r^*$ defining the interval where the proof is applicable. When this is satisfied, the radii polynomial theorem gives that
\begin{itemize}
	\item[\textbf{1.}] There exists an exact solution $\atil$ of $F(a) = 0$ in $B_{r_*}(\abar)$.
	\item[\textbf{2.}] This solution is unique in $\overline{B_{r^*}(\abar)}$.
\end{itemize}

Thus, when the radii polynomial is computed using interval arithmetic and has exactly two real non-negative roots, the zero computed numerically with the Newton iteration is close to an actual steady state of the PFC equation. Note the important fact that the ball is in $X$ so a priori, only the Fourier coefficients are controlled. Thanks to $\nu > 1$ however, we show in appendix~\ref{sm:RigorousEnergy} that this control translates into closeness in energy and in real space norms. In particular, the distance \textit{in value} between the phase fields corresponding to $\abar \midand \atil$ is at most $r_*$.

Further, we show in appendix~\ref{sm:RigorousStability} that the stability of $\atil$ in $X$ is controlled by the eigenvalues of $G$. It is important to observe that this matrix will always have a positive eigenvalue because of the trivial condition $F_{0,0} = a_{0,0} - \pbar$. This is \textit{not} indicative of instability in the context of the $\hinv$ gradient flow because $a_{0,0}$ is fixed. We shall see later that this unstable direction can be used to compute a \textit{branch of solutions} in parameter continuation. For now however, we call the number of positive eigenvalues, minus $1$, the \textit{Morse index} of $\atil$, indicating how many unstable directions are available to a given steady state for \textit{fixed} parameters.

The procedure to numerically investigate the steady states of the PFC equation is as follows:
\begin{itemize}
	\item
 Starting from a given initial condition, the Newton iteration is run until it converges up to numerical precision. 
 \item Then, the radii polynomial of the numerical guess is computed and its roots are tested. 
 \item If the proof succeeds, we can characterize an exact steady state in value, in energy and compute its stability in $X$. The parameters $(M, \nu)$ can be adjusted until the proof succeeds with a trade-off between the computational effort and closeness in $X$.
\end{itemize}

\section{Rigorous results on small domains}
\label{sec:RigorousResultsSmallDomain}
We now have a complete framework for finding verified steady states along with their energetic and stability properties. This allows us to understand the behavior of the PFC system for a given choice of $(\pbar, \beta)$, with three important caveats: 
\begin{itemize}
	\item We cannot guarantee that we have found \textit{all} steady states and therefore \textit{the} global minimizer. Indeed, we may only hope to cover a reasonable portion of the underlying space by sampling initial conditions randomly.
	\item The size of $M$ must be balanced with $\nu$ to keep $r_*$ as small as possible, keeping in mind that $r^*$ is ultimately bounded above by the distance between two steady states. In particular, large domains and large $\beta$ increase the contribution of high frequency Fourier modes, hence the truncation order can become large even for domains containing only $100$ atoms. This limits our results to small domains so our analysis is ``small scale'' in nature.
	\item The Neumann boundary conditions restrict us to a ``quadrant'' of $\hp{2}$. While the existence of a steady state, the energy bound and \textit{instability} obviously extend to $\hp{2}$, stability \textit{does not} as there may be unstable directions in the other three Fourier series that are missed by the current method.
\end{itemize}

For the last point, we sometimes observe that translational shifts have a different Morse index in $X$. This is observed for example with the stripes states, see Fig.~\ref{fig:ConnectionDiagramEnergy} (a). In this sense, we only provide a \textit{lower bound} for Morse indices in $\hp{2}$.

\subsection{Verification of the candidate minimizers}
The candidate global minimizers (constant, stripes, atoms and donuts states) introduced in appendix~\ref{sm:PFCAnsatz} have trivial Fourier coefficients by construction, given by
\begin{equation*}
\begin{split}
\text{Constant:}&\enskip a_{0,0} = \pbar \\
\text{Stripes:}&\enskip a_{0,0} = \pbar \quad a_{0,2N_y} = \frac{1}{2}A_s \\
\text{Hexagonal:}&\enskip a_{0,0} = \pbar \quad a_{N_x,N_y} = \frac{1}{2}A_h \quad a_{0,2N_y} = \frac{1}{2}A_h
\end{split}
\end{equation*}
where $A_s, A_h$ represent amplitudes that optimize the PFC energy calculation. Note that $A_h$ differs between the atoms and donuts states.

To illustrate the approach, we first applied the verification program starting at the atoms state $b$ constructed for $(\pbar, \beta) = (0.07, 0.025)$, $(N_x, N_y) = (4, 2)$ and $M = 20$. The Newton iteration was used to obtain $\bar{b}$ for which the radii polynomial was tested with $\nu = 1.05$, resulting in $r_* = 1.0 \cdot 10^{-11} \midand r^* = 6.8 \cdot 10^{-3}$. The $\ellnu$ distance between $b$ and $\bar{b}$ is $1.1 \cdot 10^{-3}$, indeed smaller than $r^*$.

The difference $b - \bar{b}$ is mainly captured by \textit{new} Fourier modes: we find that the main Fourier coefficients $b_{N_x, N_y} = b_{0, 2N_y} = -4.4 \cdot 10^{-2}$ differ by $1.5 \cdot 10^{-5}$ while the largest \textit{new} Fourier modes are $b_{8,0} = b_{4,5} = -7.4 \cdot 10^{-5}$. Moreover, the distance in the (numerical) sup norm between the two \textit{phase fields} is approximately $4.4 \cdot 10^{-4}$ which is again smaller than the $\ellnu$ distance, consistent with the $\lp{\infinity}$ bound.

This approach was repeated for the other candidates and for a few other choices of the PFC parameters in the hexagonal regime, with truncation adjusted to $\beta$. The results are presented in Table~\ref{tab:DistanceFromCandidate}, showing that such simple candidates capture well the leading behavior. Note that the agreement decreases with increasing $\beta$: compare the size of $\normnu{a-\abar}$ to $\normnu{\abar}$.

\begin{table}
\begin{center}
\caption[Steady states from ansatz]{\label{tab:DistanceFromCandidate} Data for selected values of $(\pbar, \beta)$ on the exact steady states $\atil$ near the numerical approximation $\abar$, obtained from the original candidate $a$. $M = {20, 30, 40}$ for each parameter set respectively. The Morse index was verified in $X$. We write $<\epsilon$ when the number was numerically computed as $0$. $E_0$ denotes the energy of the constant state.}
\renewcommand{\arraystretch}{0.8}
\setlength{\extrarowheight}{0.5cm}
\begin{tabular}{cccccc}
\Xhline{3\arrayrulewidth}
Ansatz	&	$(\pbar, \beta)$	&	\makecell{$\normnu{\abar}$\\$\normnu{a-\abar}$}	&	\makecell{$r_*$\\$r^*$}	&	\makecell{$E[\abar]-E_0$\\$|E[\abar]-E[\atil]|$}	&	\makecell{Morse\\index}	\\
\hline
\multirow{3}{*}[1.5em]{\makecell{\\\\Constant\\ \begin{minipage}{0.12\textwidth}\includegraphics[width=\linewidth]{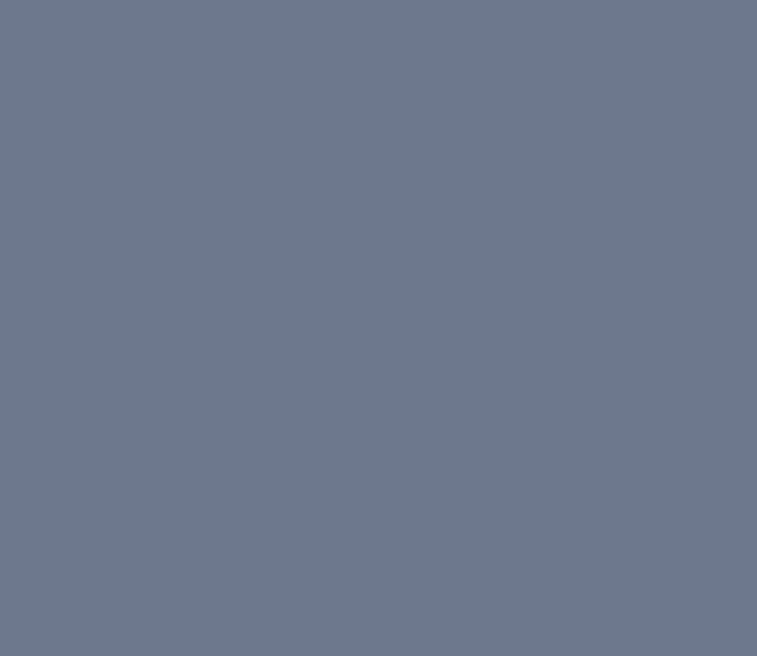}\end{minipage}}}&$(0.07, 0.025)$&\makecell{\vspace*{0.05cm}$0.07$\\$<\epsilon$}&\makecell{\vspace*{0.05cm}$4.3\cdot10^{-16}$\\$1.7\cdot10^{-2}$}&\makecell{$4.337\cdot10^{-19}$\\$4.9\cdot10^{-17}$}&4\\
\cline{2-6}
&$(0.3, 0.5)$&\makecell{\vspace*{0.05cm}$0.30$\\$<\epsilon$}&\makecell{\vspace*{0.05cm}$1.1\cdot10^{-14}$\\$1.9\cdot10^{-2}$}&\makecell{$1.388\cdot10^{-17}$\\$1.5\cdot10^{-14}$}&16\\
\cline{2-6}
&$(0.5, 1.0)$&\makecell{\vspace*{0.05cm}$0.50$\\$<\epsilon$}&\makecell{\vspace*{0.05cm}$7.2\cdot10^{-15}$\\$5.3\cdot10^{-3}$}&\makecell{$<\epsilon$\\$2.6\cdot10^{-14}$}&16\\
\hline
\multirow{3}{*}[1.5em]{\makecell{\\\\Stripes\\ \begin{minipage}{0.12\textwidth}\includegraphics[width=\linewidth]{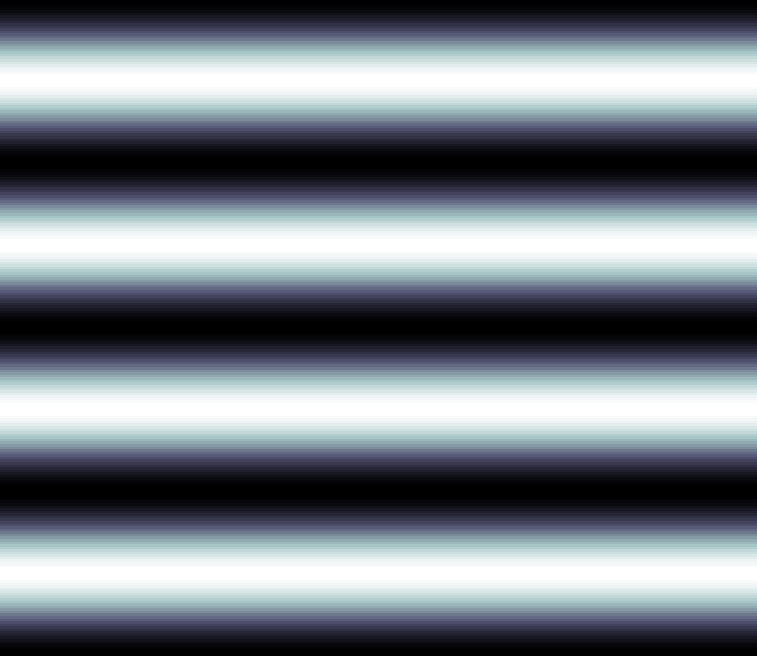}\end{minipage}}}&$(0.07, 0.025)$&\makecell{\vspace*{0.05cm}$0.21$\\$4.8\cdot10^{-4}$}&\makecell{\vspace*{0.05cm}$7.3\cdot10^{-13}$\\$5.7\cdot10^{-3}$}&\makecell{$-1.774\cdot10^{-5}$\\$4.8\cdot10^{-13}$}&0\\
\cline{2-6}
&$(0.3, 0.5)$&\makecell{\vspace*{0.05cm}$1.02$\\$4.5\cdot10^{-2}$}&\makecell{\vspace*{0.05cm}$1.9\cdot10^{-11}$\\$1.2\cdot10^{-3}$}&\makecell{$-9.369\cdot10^{-3}$\\$8.3\cdot10^{-11}$}&7\\
\cline{2-6}
&$(0.5, 1.0)$&\makecell{\vspace*{0.05cm}$1.31$\\$1.1\cdot10^{-1}$}&\makecell{\vspace*{0.05cm}$1.2\cdot10^{-12}$\\$1.7\cdot10^{-2}$}&\makecell{$-1.241\cdot10^{-2}$\\$9.1\cdot10^{-12}$}&11\\
\hline
\multirow{3}{*}[1.5em]{\makecell{\\\\Atoms\\ \begin{minipage}{0.12\textwidth}\includegraphics[width=\linewidth]{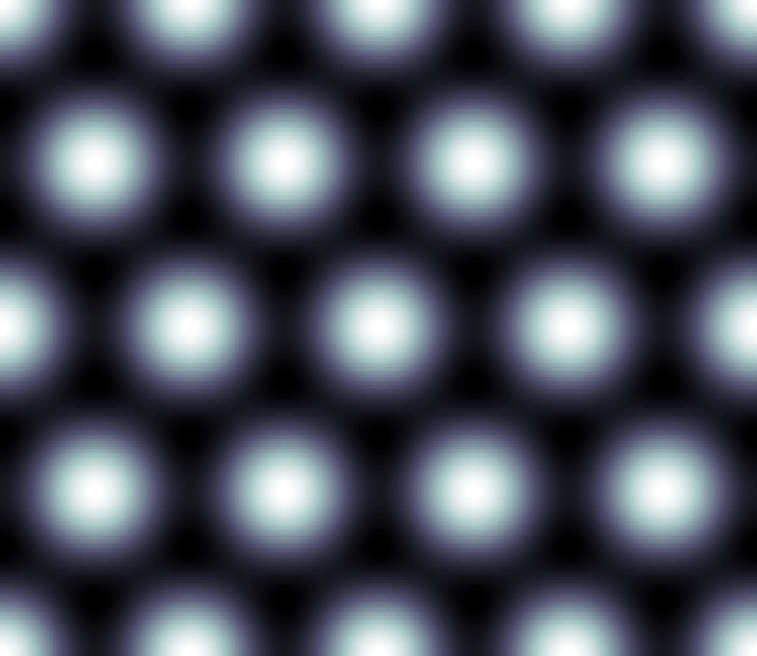}\end{minipage}}}&$(0.07, 0.025)$&\makecell{\vspace*{0.05cm}$0.41$\\$1.1\cdot10^{-3}$}&\makecell{\vspace*{0.05cm}$1.0\cdot10^{-11}$\\$6.8\cdot10^{-3}$}&\makecell{$-4.714\cdot10^{-5}$\\$1.5\cdot10^{-11}$}&0\\
\cline{2-6}
&$(0.3, 0.5)$&\makecell{\vspace*{0.05cm}$1.92$\\$8.1\cdot10^{-2}$}&\makecell{\vspace*{0.05cm}$5.5\cdot10^{-10}$\\$4.5\cdot10^{-3}$}&\makecell{$-2.089\cdot10^{-2}$\\$8.1\cdot10^{-9}$}&0\\
\cline{2-6}
&$(0.5, 1.0)$&\makecell{\vspace*{0.05cm}$2.79$\\$2.4\cdot10^{-1}$}&\makecell{\vspace*{0.05cm}$1.3\cdot10^{-11}$\\$2.9\cdot10^{-3}$}&\makecell{$-5.897\cdot10^{-2}$\\$4.4\cdot10^{-10}$}&0\\
\hline
\multirow{3}{*}[1.5em]{\makecell{\\\\Donuts\\ \begin{minipage}{0.12\textwidth}\includegraphics[width=\linewidth]{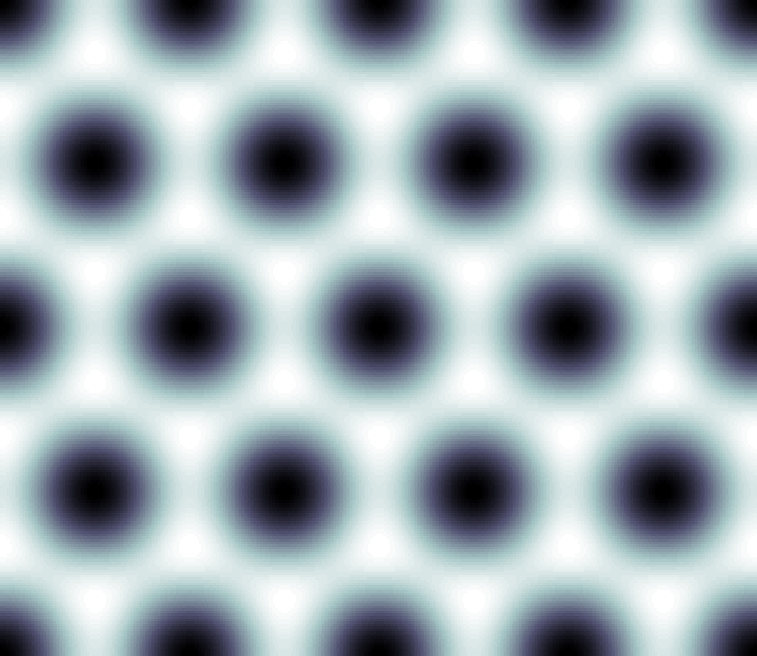}\end{minipage}}}&$(0.07, 0.025)$&\makecell{\vspace*{0.05cm}$0.19$\\$7.6\cdot10^{-4}$}&\makecell{\vspace*{0.05cm}$1.4\cdot10^{-13}$\\$3.2\cdot10^{-3}$}&\makecell{$-3.013\cdot10^{-6}$\\$8.1\cdot10^{-14}$}&3\\
\cline{2-6}
&$(0.3, 0.5)$&\makecell{\vspace*{0.05cm}$0.99$\\$8.8\cdot10^{-2}$}&\makecell{\vspace*{0.05cm}$3.7\cdot10^{-12}$\\$2.9\cdot10^{-3}$}&\makecell{$-1.839\cdot10^{-3}$\\$1.6\cdot10^{-11}$}&12\\
\cline{2-6}
&$(0.5, 1.0)$&\makecell{\vspace*{0.05cm}$1.09$\\$1.0\cdot10^{-1}$}&\makecell{\vspace*{0.05cm}$4.7\cdot10^{-12}$\\$2.1\cdot10^{-3}$}&\makecell{$-1.312\cdot10^{-3}$\\$2.1\cdot10^{-11}$}&16\\
\end{tabular}
\end{center}
\end{table}

\subsection{Steady states in the hexagonal lattice regime}
The Newton iteration can detect new steady states regardless of stability as it is based on criticality instead of minimality. This allows us to find steady states that are observed only \textit{momentarily} or even \textit{locally} during a PFC simulation. Table~\ref{tab:SteadyStatesLargeDomains} presents a few of the $28$ distinct steady states found for $(\pbar, \beta) = (0.07, 0.025)$, $(N_x, N_y) = (8, 5)$, $\nu = 1.05$ and $M = 40$. Starting at random initial coefficient matrices, the Newton iteration converges in 15 to 50 steps. The four main ansatz were also explicitly tested, as only the atoms state could be reached from random initial conditions.

\begin{table}[h!]
\begin{center}
\caption[Steady states for large domains]{\label{tab:SteadyStatesLargeDomains} Data on steady states for $(\pbar, \beta) = (0.07, 0.025)$ and $(N_x, N_y) = (8, 5)$, capturing roughly $80$ atoms. The observed count is the number of times the steady state, including its discrete translational shifts, were reached out of $200$ randomized trials.}
\renewcommand{\arraystretch}{0.8}
\begin{tabular}{cccccc}
\Xhline{3\arrayrulewidth}
Visualization	&	\makecell{$r_*$\\$r^*$}	&	\makecell{$E[\abar]-E_0$\\$|E[\abar]-E[\atil]|$}	&	Morse index		&	Count	\\
\hline
\begin{minipage}{0.12\textwidth}\vspace*{0.05cm}\includegraphics[width=\linewidth]{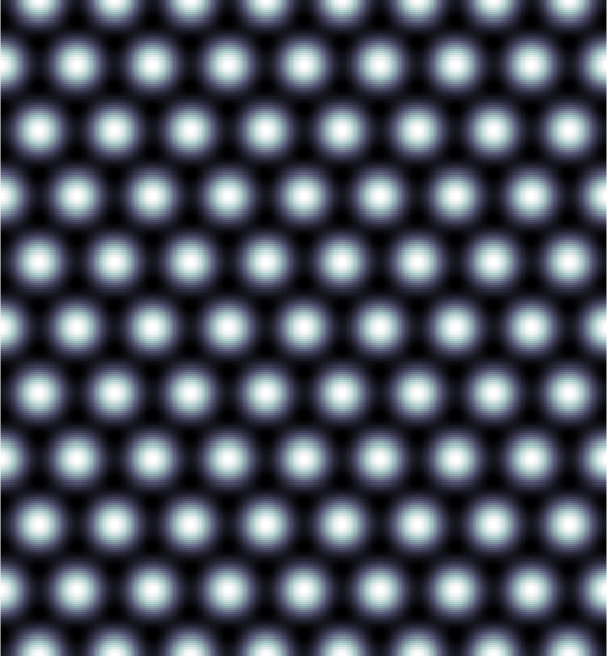}\vspace*{0.05cm}\end{minipage}&\makecell{$1.0\cdot10^{-8}$\\$2.8\cdot10^{-3}$}&\makecell{$-4.714\cdot10^{-5}$\\$2.2\cdot10^{-8}$}&$0$&$53$\\
\hline
\begin{minipage}{0.12\textwidth}\vspace*{0.05cm}\includegraphics[width=\linewidth]{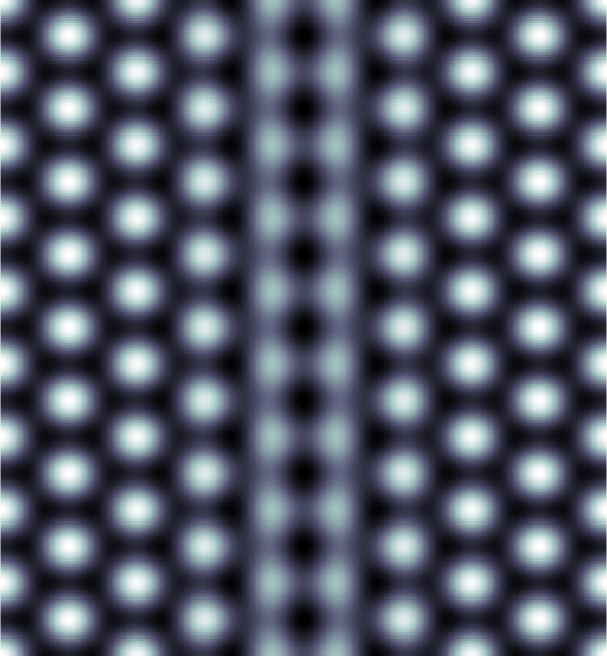}\vspace*{0.05cm}\end{minipage}&\makecell{$2.5\cdot10^{-8}$\\$2.9\cdot10^{-4}$}&\makecell{$-2.358\cdot10^{-5}$\\$6.7\cdot10^{-8}$}&$0$&$21$\\
\hline
\begin{minipage}{0.12\textwidth}\vspace*{0.05cm}\includegraphics[width=\linewidth]{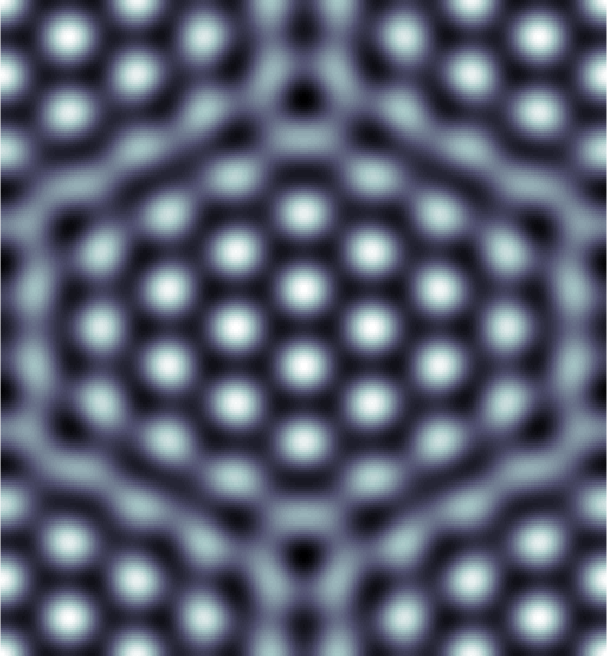}\vspace*{0.05cm}\end{minipage}&\makecell{$5.4\cdot10^{-8}$\\$5.8\cdot10^{-5}$}&\makecell{$-2.215\cdot10^{-5}$\\$1.9\cdot10^{-7}$}&$0$&$57$\\
\hline
\begin{minipage}{0.12\textwidth}\vspace*{0.05cm}\includegraphics[width=\linewidth]{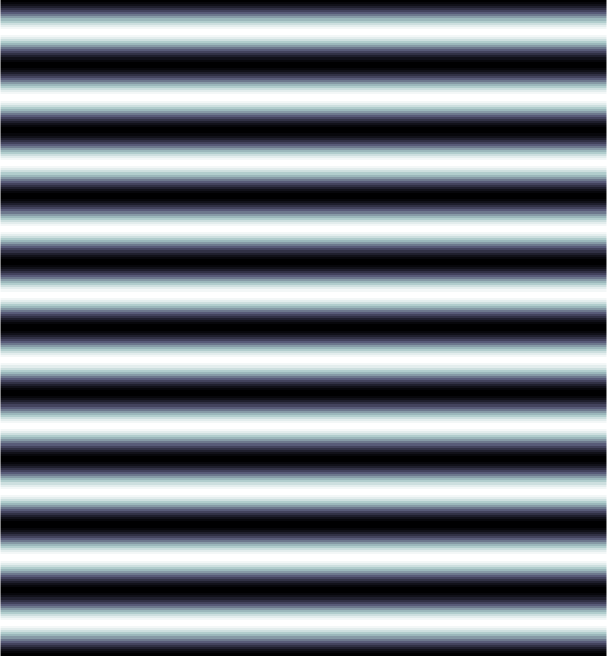}\vspace*{0.05cm}\end{minipage}&\makecell{$1.6\cdot10^{-9}$\\$2.0\cdot10^{-3}$}&\makecell{$-1.774\cdot10^{-5}$\\$1.3\cdot10^{-9}$}&$1$&$0$\\
\hline
\begin{minipage}{0.12\textwidth}\vspace*{0.05cm}\includegraphics[width=\linewidth]{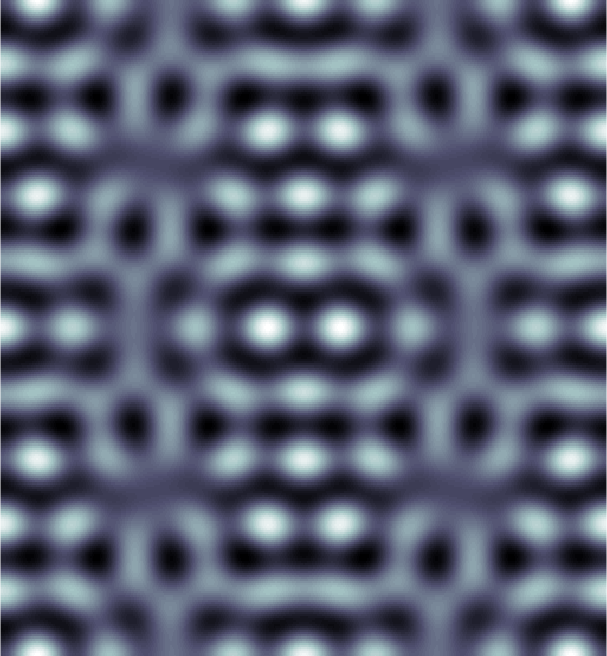}\vspace*{0.05cm}\end{minipage}&\makecell{$3.5\cdot10^{-8}$\\$5.4\cdot10^{-5}$}&\makecell{$-1.161\cdot10^{-5}$\\$9.0\cdot10^{-8}$}&$2$&$1$\\
\end{tabular}
\end{center}
\end{table}

Note that the energy of the exact steady states can be compared from Table~\ref{tab:SteadyStatesLargeDomains}: for instance, the energy of the \textit{exact} atoms state is bounded away from the others so it is guaranteed to be the best candidate global minimizer out of the observed steady states at the current parameter values.

The second and third states presented in the table clearly display two grains of the same orientation but with boundary atoms meeting ``head-to-head.'' This is essentially an intermediate in the grains slipping on one another that is stabilized by the restrictions of the boundary conditions. Such states then represent a grain boundary that is stable, at least in $X$. When PFC simulations~\cite{ELSEY_Scheme} are initialized at these states, the flow appears to be stable for thousands of steps then suddenly goes to the hexagonal lattice, meaning there are unstable directions in the rest of $\hp{2}$. Nevertheless, the fact remains that grain boundaries can be \textit{steady states}.

\subsection{Steady states in the localized patterns regime}
Table~\ref{tab:SteadyStatesLocalized} presents some steady states found for $(\pbar, \beta) = (0.5, 0.6)$, $(N_x, N_y) = (7, 4)$, $\nu = 1.01$ and $M = 60$. In this regime, localized or coexistence patterns are observed in PFC simulations, some of which we can confirm to be steady states: note in particular the existence of a ``single atom'' state. We see here that the global minimizer cannot be of the four main ansatz. We observe \textit{two} atoms states with different amplitudes and stability, highlighting the fact that the ``linear'' candidate is no longer appropriate as $\beta$ increases and nonlinear effects begin to dominate the energy.

Similar results have been obtained previously for a version of Swift-Hohenberg with broken $\psi \to -\psi$ symmetry, see~\cite{LLOYD_LocalizedHexagons, VANDENBERG_CoexistenceHexagonsRolls}.

\begin{table}[H]
\begin{center}
\caption[Steady states for localized patterns]{\label{tab:SteadyStatesLocalized} Data on steady states for $(\pbar, \beta) = (0.5, 0.6)$ and $(N_x, N_y) = (7, 4)$. No count is provided because only a few trials were attempted.}
\renewcommand{\arraystretch}{0.8}
\begin{tabular}{cccc}
\Xhline{3\arrayrulewidth}
Visualization	&	\makecell{$r_*$\\$r^*$}	&	\makecell{$E[\abar]-E_0$\\$|E[\abar]-E[\atil]|$}	&	Morse index	\\
\hline
\begin{minipage}{0.12\textwidth}\vspace*{0.05cm}\includegraphics[width=\linewidth]{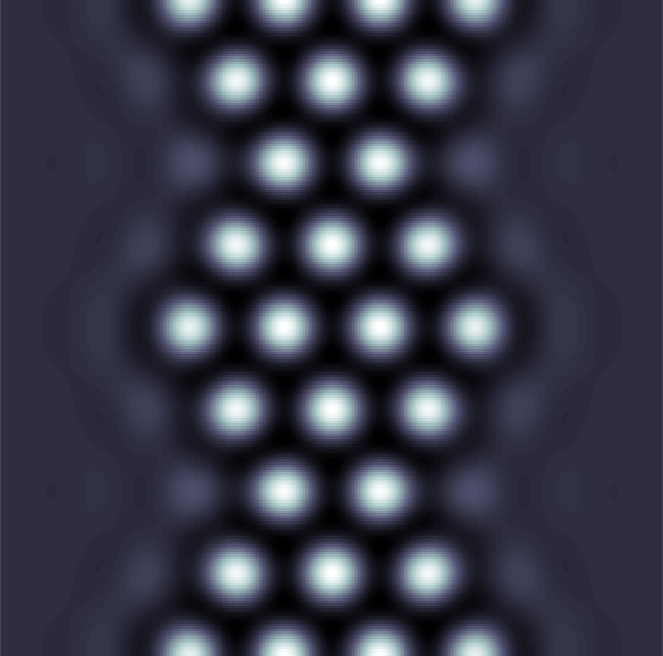}\vspace*{0.05cm}\end{minipage}&\makecell{$9.9\cdot10^{-11}$\\$3.3\cdot10^{-4}$}&\makecell{$-2.465\cdot10^{-3}$\\$4.7\cdot10^{-9}$}&$0$\\
\hline
\begin{minipage}{0.12\textwidth}\vspace*{0.05cm}\includegraphics[width=\linewidth]{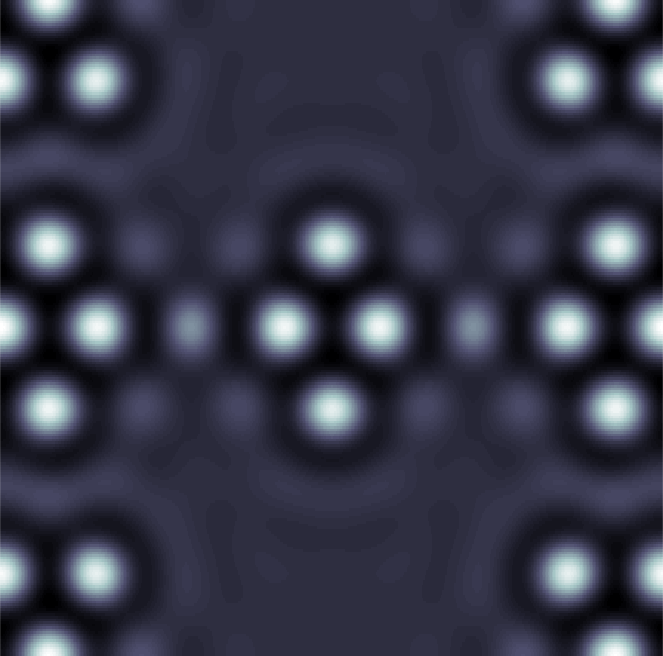}\vspace*{0.05cm}\end{minipage}&\makecell{$1.6\cdot10^{-10}$\\$2.4\cdot10^{-4}$}&\makecell{$-1.457\cdot10^{-3}$\\$1.9\cdot10^{-8}$}&$1$\\
\hline
\begin{minipage}{0.12\textwidth}\vspace*{0.05cm}\includegraphics[width=\linewidth]{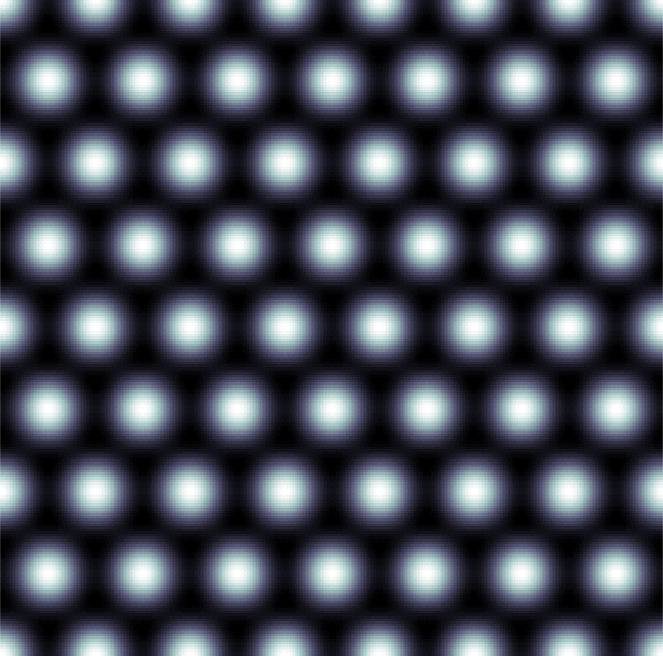}\vspace*{0.05cm}\end{minipage}&\makecell{$2.6\cdot10^{-10}$\\$7.7\cdot10^{-5}$}&\makecell{$-2.670\cdot10^{-4}$\\$2.3\cdot10^{-9}$}&$8$\\
\hline
\begin{minipage}{0.12\textwidth}\vspace*{0.05cm}\includegraphics[width=\linewidth]{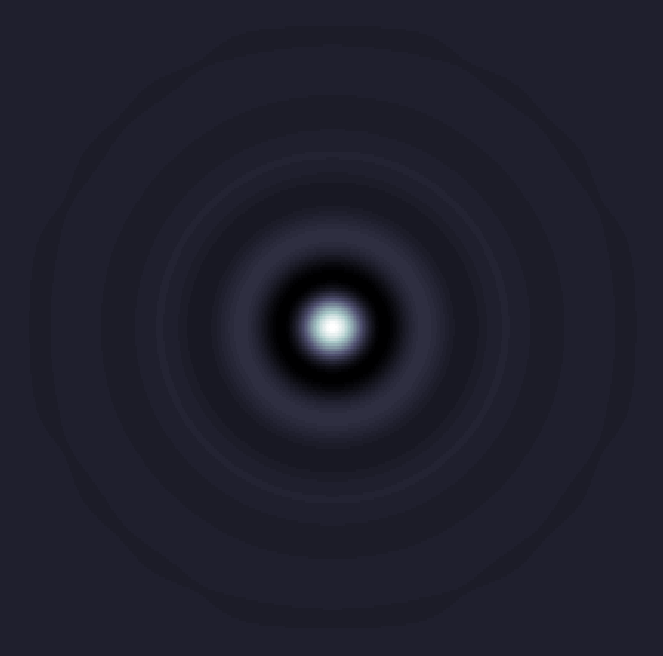}\vspace*{0.05cm}\end{minipage}&\makecell{$2.1\cdot10^{-11}$\\$1.3\cdot10^{-3}$}&\makecell{$-6.420\cdot10^{-5}$\\$3.0\cdot10^{-10}$}&$0$\\
\hline
\begin{minipage}{0.12\textwidth}\vspace*{0.05cm}\includegraphics[width=\linewidth]{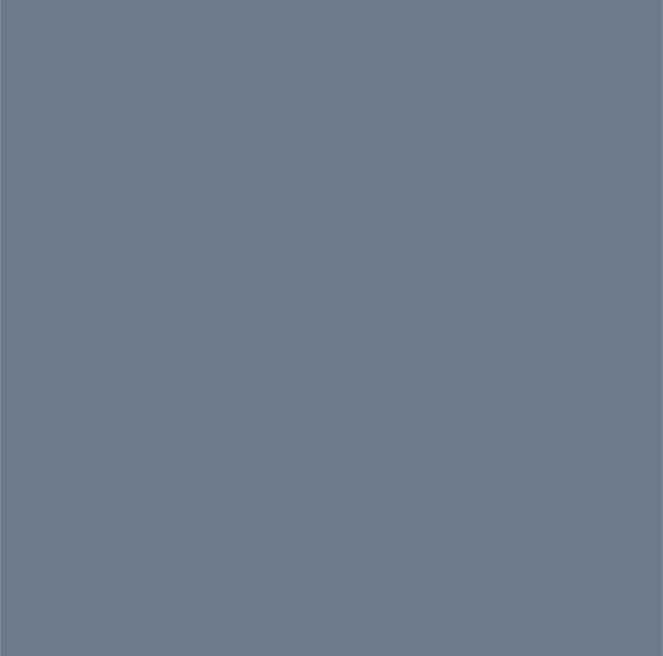}\vspace*{0.05cm}\end{minipage}&\makecell{$9.3\cdot10^{-15}$\\$4.8\cdot10^{-2}$}&\makecell{$<\epsilon$\\$2.8\cdot10^{-14}$}&$0$\\
\hline
\begin{minipage}{0.12\textwidth}\vspace*{0.05cm}\includegraphics[width=\linewidth]{Graphics/SS_Localized_7.png}\vspace*{0.05cm}\end{minipage}&\makecell{$1.7\cdot10^{-11}$\\$1.6\cdot10^{-3}$}&\makecell{$5.629\cdot10^{-4}$\\$5.7\cdot10^{-11}$}&$20$\\
\end{tabular}
\end{center}
\end{table}

\subsection{Steady states for the large $\beta$ regime}
Table~\ref{tab:SteadyStatesLargeBeta} shows a selection of steady states found in the large $\beta$ regime, $(\pbar, \beta) = (2.5, 20.0)$, $(N_x, N_y) = (4, 2)$, $\nu = 1.01$ and $M = 65$. In this regime, the microscopic organization is lost as constant patches of phase form, with value close to $\pm \sqrt{\beta}$, meaning that the double well term of the PFC functional dominates the oscillation term.

\begin{table}[H]
\begin{center}
\caption[Steady states for large $\beta$]{\label{tab:SteadyStatesLargeBeta} Data on steady states for $(\pbar, \beta) = (2.5, 20.0)$ and $(N_x, N_y) = (4, 2)$.}
\renewcommand{\arraystretch}{0.8}
\begin{tabular}{cccc}
\Xhline{3\arrayrulewidth}
Visualization	&	\makecell{$r_*$\\$r^*$}	&	\makecell{$E[\abar]-E_0$\\$|E[\abar]-E[\atil]|$}	&	Morse index	\\
\hline
\begin{minipage}{0.12\textwidth}\vspace*{0.05cm}\includegraphics[width=\linewidth]{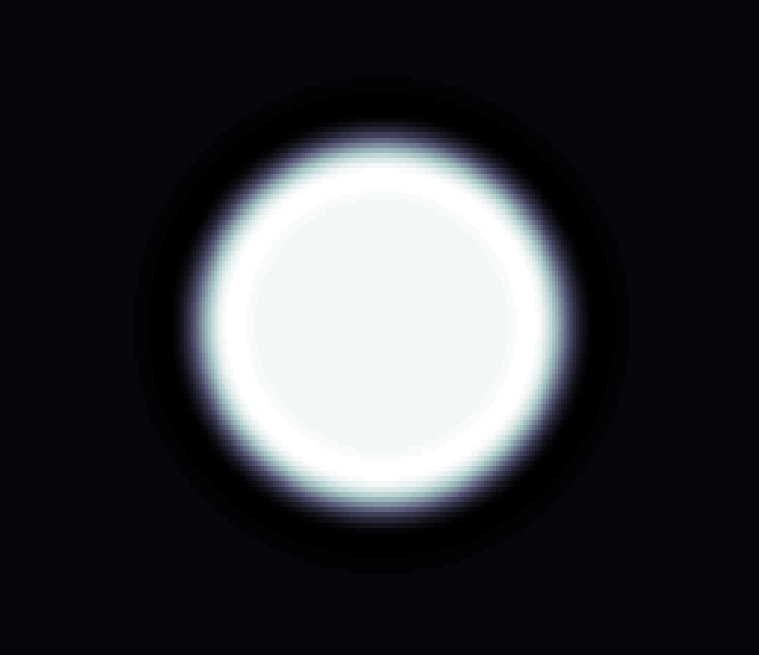}\vspace*{0.05cm}\end{minipage}&\makecell{$1.5\cdot10^{-8}$\\$5.0\cdot10^{-5}$}&\makecell{$-36.71$\\$5.8\cdot10^{-4}$}&$0$\\
\hline
\begin{minipage}{0.12\textwidth}\vspace*{0.05cm}\includegraphics[width=\linewidth]{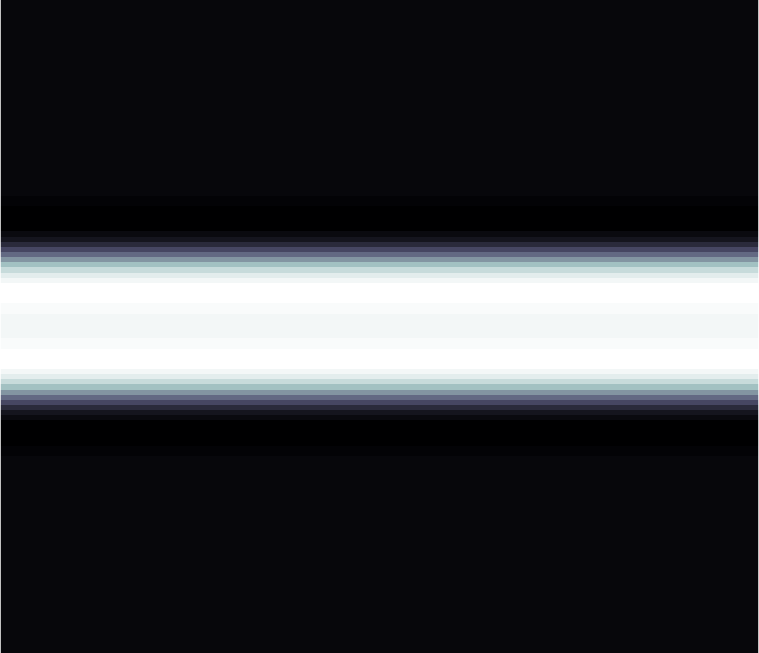}\vspace*{0.05cm}\end{minipage}&\makecell{$1.2\cdot10^{-9}$\\$2.6\cdot10^{-4}$}&\makecell{$-35.48$\\$3.5\cdot10^{-6}$}&$0$\\
\hline
\begin{minipage}{0.12\textwidth}\vspace*{0.05cm}\includegraphics[width=\linewidth]{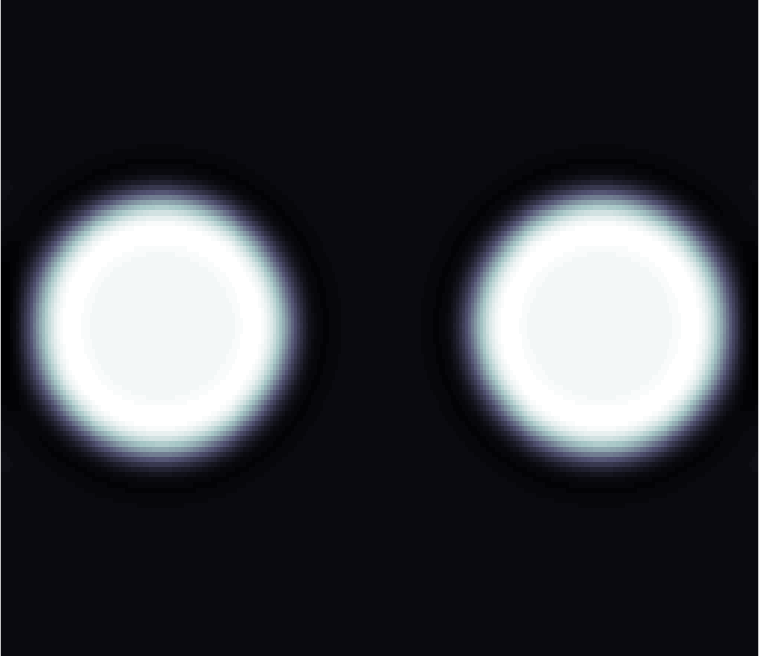}\vspace*{0.05cm}\end{minipage}&\makecell{$2.3\cdot10^{-8}$\\$2.9\cdot10^{-5}$}&\makecell{$-35.08$\\$1.0\cdot10^{-3}$}&$0$\\
\hline
\begin{minipage}{0.12\textwidth}\vspace*{0.05cm}\includegraphics[width=\linewidth]{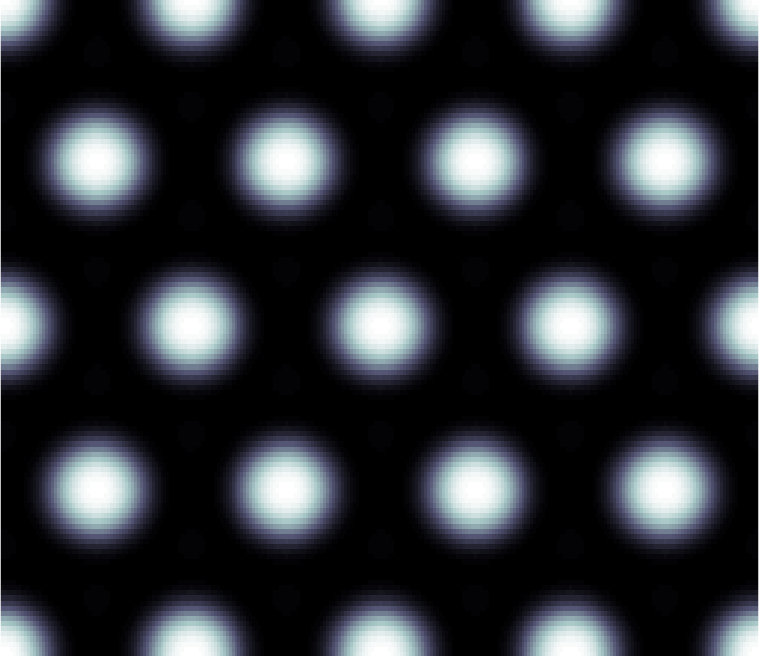}\vspace*{0.05cm}\end{minipage}&\makecell{$1.2\cdot10^{-7}$\\$4.4\cdot10^{-6}$}&\makecell{$-24.11$\\$4.6\cdot10^{-4}$}&$13$\\
\end{tabular}
\end{center}
\end{table}

\subsection{Phase diagram with verified steady states}
\label{sec:ImprovedPhaseDiagram}

The framework allows us to construct a ``rigorous'' phase diagram for PFC. Here the adjective ``rigorous" does  {\bf not} mean that we have identified the  ground state; but rather that the respective candidate state has been rigorously verified in its parameter regime. 
To this end, one must construct a ``patchwork'' of $(\pbar, \beta)$ split in regions in which we have a proof that a given state is a global minimizer. For now, we restrict ourselves to proving that one of the steady states near the known candidate minimizers has lower energy than all other \textit{known} steady states at given \textit{points}. Further, our attempt is somewhat limited by the small domains we can access. Nevertheless, this construction is useful and does indicate rigorously where the candidates \textit{cannot} be global minimizers.

Our approach is as follows: we discretize the $(\pbar, \beta)$ parameter space to some desired accuracy and for each point, we test the four ansatz and several other candidates obtained from random initial coefficients. When one of the four ansatz has verified lower energy than the others, up to translational symmetries, we label that point accordingly and otherwise leave the point blank. Fig.~\ref{fig:RN_PhaseDiagram} (a) shows the resulting diagram for small parameter values with $(N_x, N_y) = (4, 2)$, $\nu = 1.01$, $M = 20$. At each point, $30$ trials of the Newton iteration were tried and verified rigorously. Note that the points below $\beta = \pbar^2$ could have been skipped since the constant state is known to be the global minimizer in that regime~\cite{SHIROKOFF_GlobalMinimality}. This diagram matches the one obtained in the appendices with linear stability analysis.

Fig.~\ref{fig:RN_PhaseDiagram} (b) shows the phase diagram near $(\pbar, \beta) = (0.5, 0.6)$ where localized patterns have been observed. The domain is the same size but $M = 30$ to accommodate the larger $\beta$. At each point, $15$ trials were tried and verified, leading to points that have lower energy than the atoms or constant states. This indeed shows the existence of a region where localized patterns are more energetically favourable. This region gives an estimate of the full coexistence region that ultimately cannot be made explicit without more refined techniques.
\begin{figure}
	\centering
	\subfloat[\hspace*{-1.0cm}]{\includegraphics[width=0.48\textwidth]{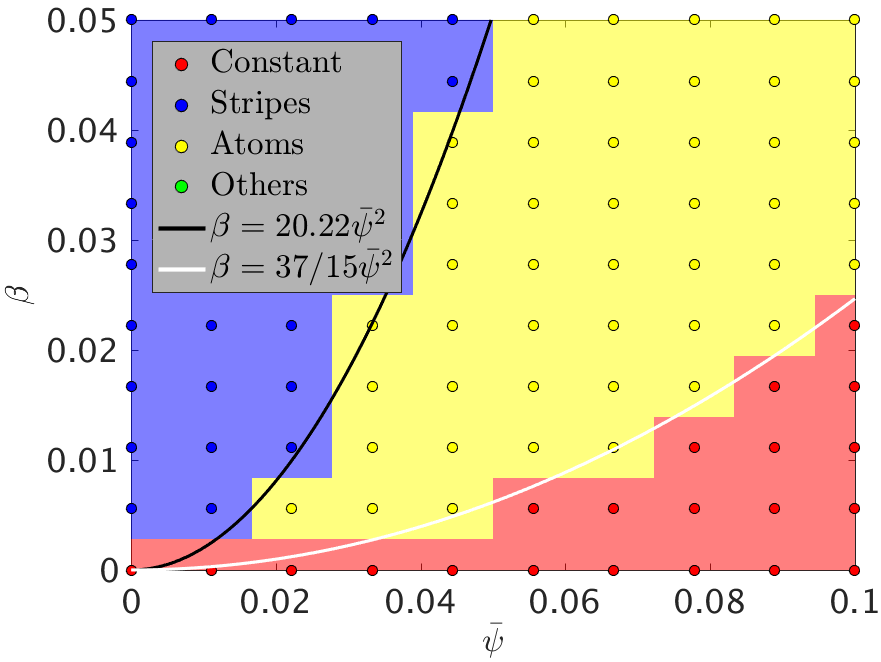}}\enskip
	\subfloat[\hspace*{-1.05cm}]{\includegraphics[width=0.48\textwidth]{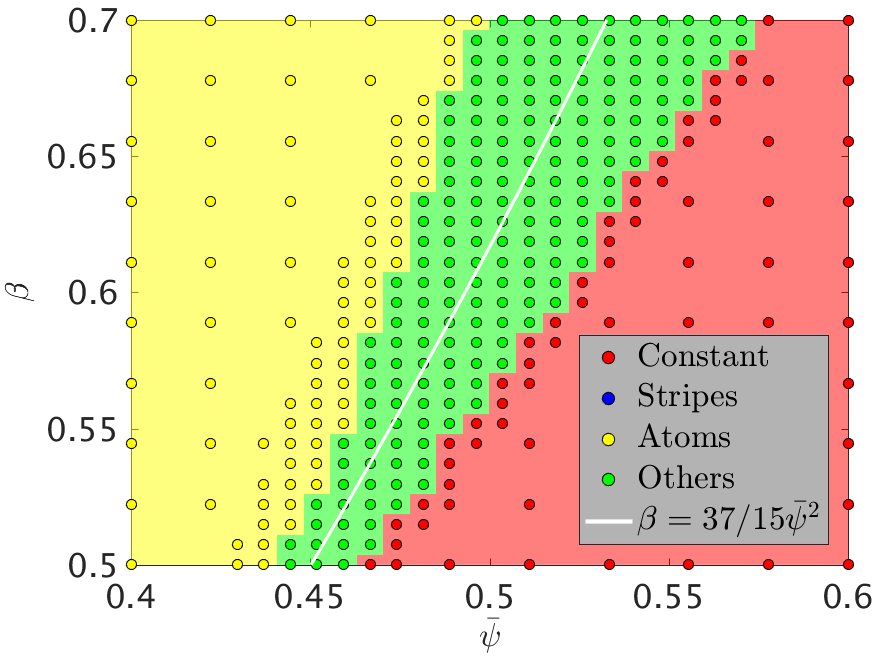}}
	
	\caption[Rigorous computation of the phase diagram]{Phase diagram for small parameter values (a) and for the localized patterns regime. (b) All points are prepared by rigorously verifying that the exact steady state around the ansatz have lower energy than all other observed steady states, up to translational shifts. Colored regions are filled in to guide the eye. The curves show the condition for the energy of the basic ansatz to be equal.}
	\label{fig:RN_PhaseDiagram}
\end{figure}

\subsection{Rigorous results for two-mode PFC}
\label{sec:RigorousTwoMode}
As a final example, Table~\ref{tab:SteadyStatesTwoMode} shows three verified steady states for two-mode PFC with $q = 1/\sqrt{2}$, $(\pbar, \beta) = (0.09, 0.025)$, $(N_x, N_y) = (12, 4)$, $\nu = 1.01$ and $M = 64$. Note that here, $L_x = 2\sqrt{2}\pi N_x \midand L_y = 2\sqrt{2}\pi L_y$ to fit the symmetry of the square lattice. The second state shows two grains slipping on each other; in contrast, especially to the result for hexagonal lattices, the third state is a grain boundary with \textit{non-zero} misorientation. Here, the rectangular domains with Neumann boundary conditions can support the geometry of the square lattice at $0^\circ$ \textit{and} $45^\circ$ rotations, so we can observe their coexistence. Since this result can be extended to larger domains by simple tiling operations, we conclude that straight grain boundaries \textit{can} be steady states even in infinite domains where boundary conditions cannot ``help'' stabilizing such defects.

Moreover, this grain boundary was observed to be (numerically) stable in two-mode PFC simulations in the sense that small random perturbations of the phase field always converged back to the grain boundary state. This is not a rigorous proof of stability in $\hp{2}$, but it gives a good indication that grain boundaries are likely to be stable features in the PFC model.

\begin{table}
\begin{center}
\caption[Steady states for two-mode PFC]{\label{tab:SteadyStatesTwoMode} Data on steady states for $(\pbar, \beta) = (0.09, 0.025)$ and $(N_x, N_y) = (12, 4)$ in the two-mode PFC model with $q = 1/\sqrt{2}$. $E_0$ is the energy of the constant state for two-mode PFC. $E[\abar]$ is listed for comparison purposes but it is not rigorously bounded.}
\renewcommand{\arraystretch}{0.8}
\begin{tabular}{cccc}
\Xhline{3\arrayrulewidth}
Visualization	&	\makecell{$r_*$\\$r^*$}	&	$E[\abar]-E_0$	&	Morse index	\\
\hline
\begin{minipage}{0.36\textwidth}\vspace*{0.05cm}\includegraphics[width=\linewidth]{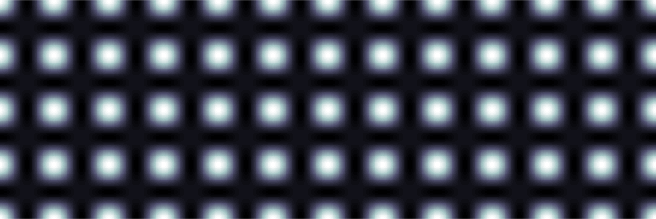}\vspace*{0.05cm}\end{minipage}&\makecell{$2.0\cdot10^{-12}$\\$3.6\cdot10^{-4}$}&$-2.758\cdot10^{-5}$&$0$\\
\hline
\begin{minipage}{0.36\textwidth}\vspace*{0.05cm}\includegraphics[width=\linewidth]{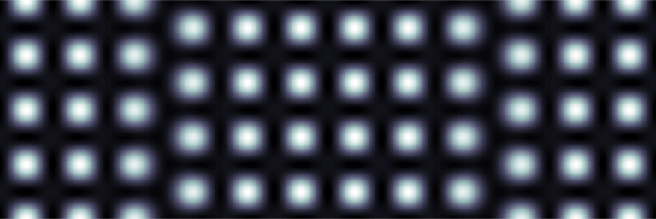}\vspace*{0.05cm}\end{minipage}&\makecell{$6.5\cdot10^{-12}$\\$1.3\cdot10^{-4}$}&$-2.319\cdot10^{-5}$&$0$\\
\hline
\begin{minipage}{0.36\textwidth}\vspace*{0.05cm}\includegraphics[width=\linewidth]{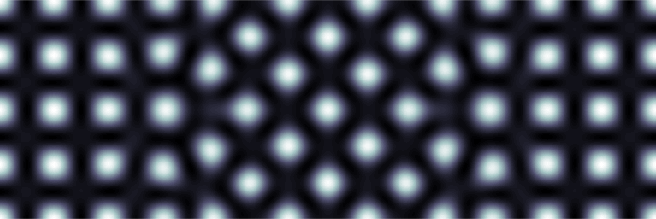}\vspace*{0.05cm}\end{minipage}&\makecell{$2.4\cdot10^{-11}$\\$4.2\cdot10^{-5}$}&$-2.244\cdot10^{-5}$&$0$\\
\end{tabular}
\end{center}
\end{table}

\section{Connections between steady states}
Suppose $\Psi_1, \Psi_2$ represent two steady states, we say that there is a connection (or a connecting orbit) from $\Psi_1$ to $\Psi_2$ if there exists a solution $\psi(t)$ with the property that
\begin{equation*}
\lim_{t \to -\infinity} \psi(t) = \Psi_1 \midand \lim_{t \to +\infinity} \psi(t) = \Psi_2 \eqdot
\end{equation*}

More precisely, the connecting orbit leaves the unstable manifold of $\Psi_1$ and ends up in the stable manifold of $\Psi_2$. Since the PFC equation is a gradient flow, there cannot exist non-trivial homoclinic connections so there is a natural \textit{hierarchy} of steady states expressed through heteroclinic connections. This concept is extremely useful to ``visualize'' the energy landscape.

States with Morse index $0$ are stable (for fixed parameters) and are thus at the bottom of the hierarchy. Those states with Morse index $1$ have one unstable direction, so there are two distinct perturbations that lead away from the state. For states with Morse index $2$, two unstable directions span infinitely many such perturbations, and so on. To detect connections, we propose to initialize a PFC flow near an unstable steady state offset by such perturbations. If the flow becomes close enough to another known steady states, we stop and propose a \textit{conjectured} connection between the two steady states. This procedure often allows us to find unknown steady states: when the flow stagnates, the Newton iteration can be run and often converges in very few steps to a steady state that can be verified. Alternatively, we could check for inclusion in the target $r^*$ ball, but this is a very restrictive criterion that limits our numerical investigation, especially when obtaining connections to unstable states. We use the PFC scheme detailed in~\cite{ELSEY_Scheme}.

While we cannot for the moment prove such claims because ``parameterizing'' the infinite dimensional stable manifold of the unstable steady states is highly non-trivial, we are aware of some preliminary work in this direction \cite{BERG_JAQUETTE_MIRELES}.
That said, computer-assisted proofs of connecting orbits from saddle points to asymptotically stable steady states in parabolic PDEs are starting to appear \cite{MR3773757,MR3906120,JAQUETTE_LESSARD_TAKAYASU}.

We first consider the standard parameters $(\pbar, \beta) = (0.07, 0.025)$ and use the very small domain $(N_x, N_y) = (2, 1)$. This choice is made to ensure that the constant state has Morse index $2$ in $X$ to simplify the visualization. We find seven steady states: both possible translations of the atoms, stripes and donuts state, and the trivial constant state. Following the program described above, we can construct the ``connection diagram'' shown in Fig.~\ref{fig:ConnectionDiagramEnergy} (a) with the arrows indicating that a connection was found from a state to the other. Note in particular that the stable stripes state to the right \textit{numerically} decays into the appropriately shifted hexagonal lattices, but this is a slow process as the sine modes must grow out of numerical noise. This clearly shows that our method cannot be used to guarantee stability in $\hp{2}$ because it \textit{cannot} depend on translational shifts.
\begin{figure}
	\centering
	\subfloat[]{\includegraphics[width=0.65\textwidth]{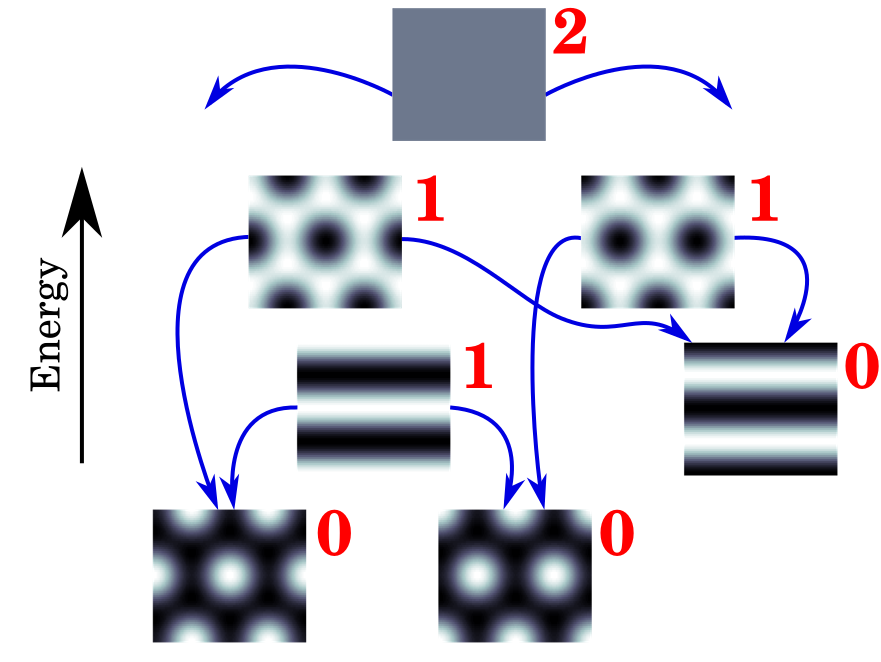}}\enskip
	\subfloat[]{\includegraphics[width=0.65\textwidth]{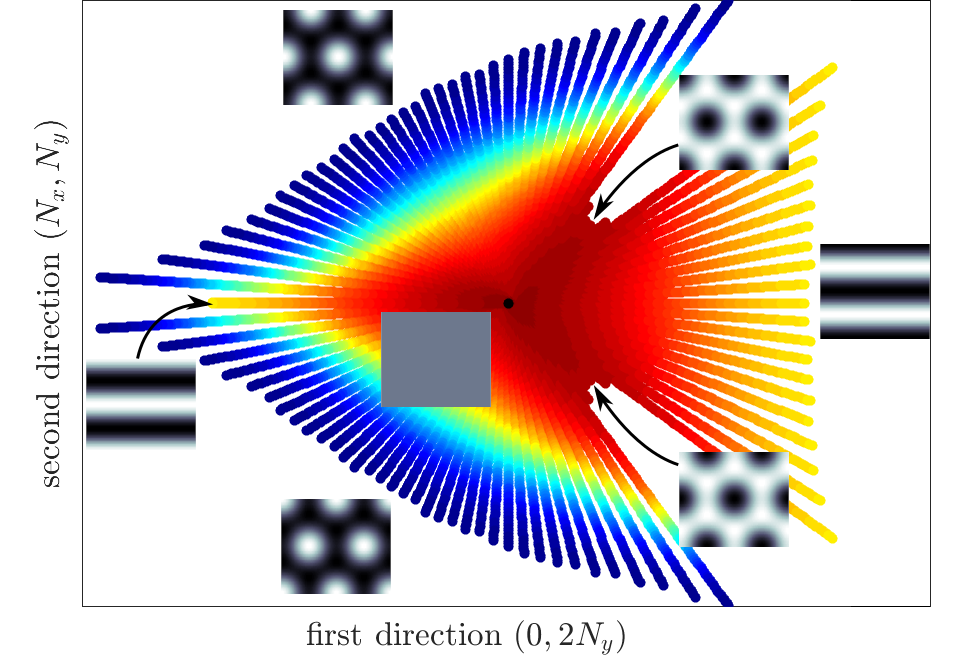}}
		
	\caption[Connection energy diagram]{Connection diagram (a) where arrows represent likely connections; the constant state is connected to all others. The vertical axis gives the ordering in energy while the numbers give the Morse index. (b) Energy visualization with respect to the unstable directions of the constant state. This diagram illustrates how the unstable directions combine to transform the constant state into other lower energy states. The unstable directions serve as the main axes and the lines represent different initial perturbations. The length of the lines indicate the number of PFC steps before the flows becomes close to the connecting steady states. Colors represent energy (red for high and blue for low energy).}
	\label{fig:ConnectionDiagramEnergy}
\end{figure}

We also propose a visualization method for such diagrams shown in Fig.~\ref{fig:ConnectionDiagramEnergy} (b). Take for example the constant state with its two unstable directions given by the coefficients $a_{0, 2N_y}$ and $a_{N_x, N_y}$. We place the constant state at the origin and plot radial lines along linear combination of the unstable directions. The line length corresponds to the number of PFC steps needed to approach the target steady states. In addition, we can color the points along the line as a function of energy to indicate energetic relationships. A variant would be to show the energy as the $z$-component of a surface; essentially giving an indirect visualization of the energy landscape through 2D unstable manifolds. In particular, this diagram clarifies the relationships between the steady states. For instance, the stripes states are formed by adding the $a_{0, 2N_y}$ mode to the constant state while the donuts are combinations of the atoms and stripes states.

We now consider the localized patterns regime to illustrate these ideas with states of high Morse index. We do not attempt to build a higher dimensional visualization, but simply attempt to recover the ``pathways'' between the highly unstable hexagonal lattice with Morse index $20$ towards stable steady states. This is visualized in the connection diagram of Fig.~\ref{fig:ConnectionDiagramEnergyLocalized} (a) which includes a few states of Table~\ref{tab:SteadyStatesLocalized}. In (b), we plot the energy along the PFC flow starting from the index 2 state; this plot can be thought of as one of the rays in a diagram like Fig.~\ref{fig:ConnectionDiagramEnergy} (b). Note that along the flow, the energy decreases in ``steps'' corresponding to changes in topology, i.e. the formation (or removal) of atoms. In this process, we could not verify that these intermediates are steady states since the Newton iteration always converged to the endpoint; we then suppose they are short-lived ``metastable'' states.
\begin{figure}
	\centering
	\subfloat[]{\includegraphics[width=0.45\textwidth]{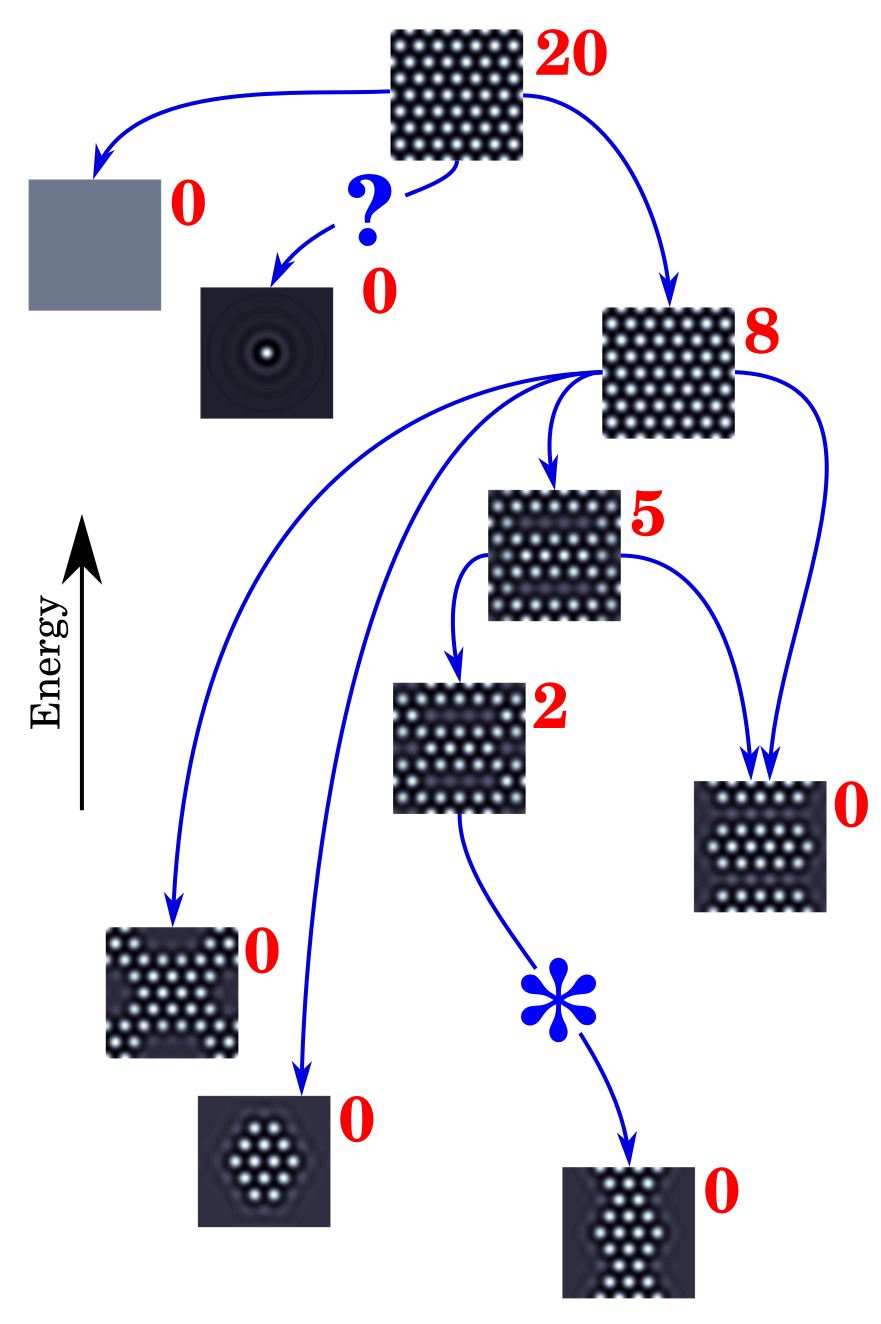}}\enskip
	\subfloat[]{\includegraphics[width=0.45\textwidth]{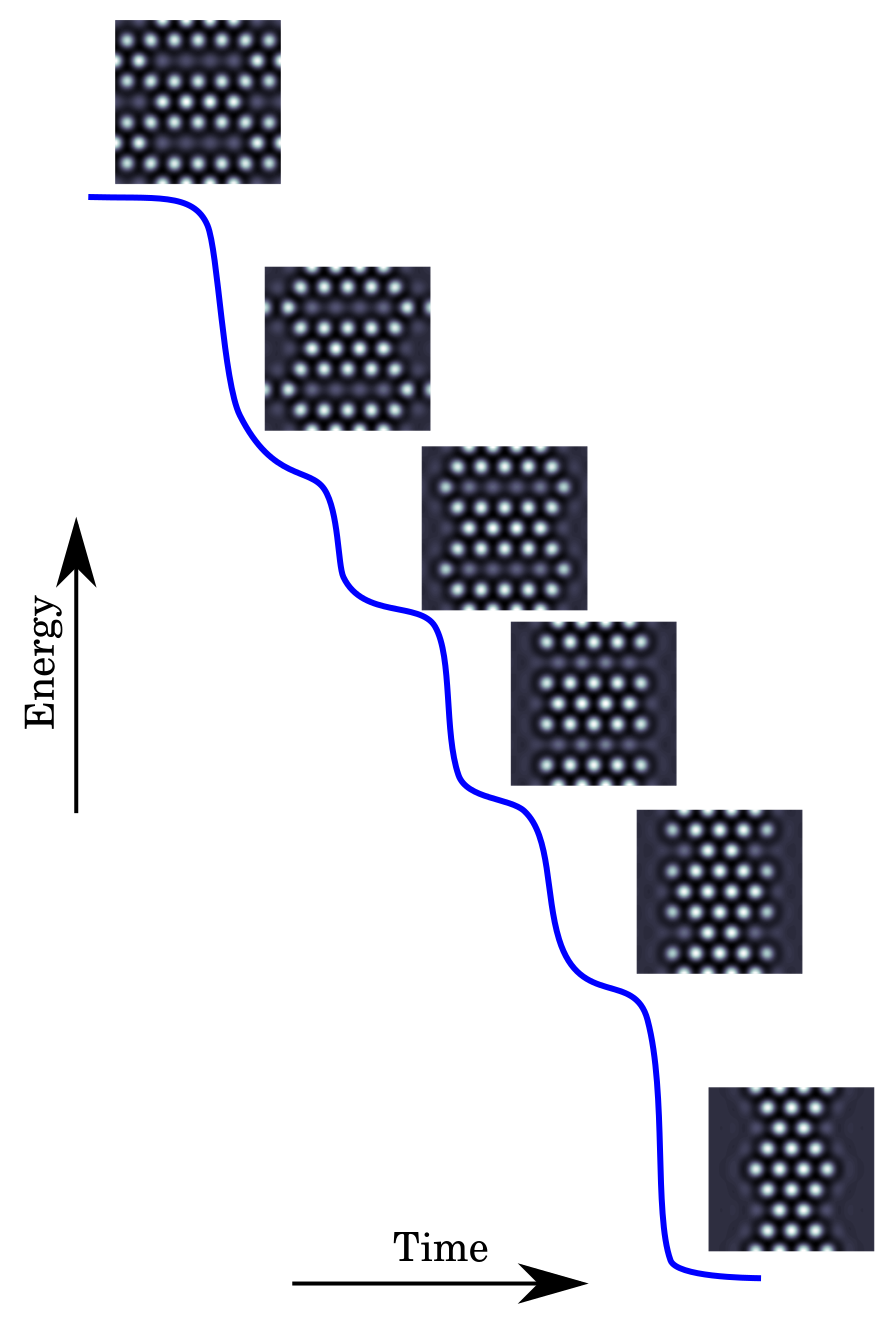}}
		
	\caption[Connection energy diagram]{Connection diagram (a) where arrows represent a few of the connections found. The two hexagonal lattice states differ in their amplitude and stability. The vertical axis roughly indicates the energy while the numbers give the Morse indices. We could not obtain (nor disprove) a connection to the single atom state, indicated with the question mark. The connection labeled with a star is broken down in the energy plot to the right (b). These states appear to be metastable intermediates where the energy gradient becomes small and the evolution slows down considerably. The blue curve shows the energy as a function of time in arbitrary units, highlighting momentaneous ``flats'' in the evolution.}
	\label{fig:ConnectionDiagramEnergyLocalized}
\end{figure}

It is difficult to obtain perturbations that can flow to \textit{desired} steady states, especially when they are unstable; see how only a few directions reach the Morse index $1$ states in Fig.~\ref{fig:ConnectionDiagramEnergy} (b). Indeed, unless ``trivial'' combinations of the unstable direction happen to go to an unstable state, we are unlikely to find such connections numerically. Similarly, our attempts to find a perturbation that connects the starting lattice to the single atom state were unfruitful.

\section{Parameter continuation for steady states}
A verified steady state $\atil$ for some parameter $(\pbar, \beta)$ is usually part of a family of steady states representing a ``phase'' of matter. In fact, the candidate minimizers defined in appendix~\ref{sm:PFCAnsatz} as functions of $(\pbar, \beta)$ approximate such families, or branches in the \textit{bifurcation diagram}. In this context, we can construct such branches by starting at a known steady state, vary $\pbar$ and find the closest steady state at this new parameter value. 

Several verified techniques exist for following branches, see~\cite{VANDENBERG_BranchFollowing} and~\cite{VANDENBERG_OhtaKawasaki} for an application to Ohta-Kawasaki. We use non-verified pseudo-arclength continuation~\cite{MR910499} in $\pbar$. Note that the unstable direction that is to followed is precisely given by the one corresponding to the ``fixed'' $a_{0,0} = \pbar$ condition and this is one of the reasons that we chose to enforce this directly in the formulation of $F$. As a possible extension, 2D manifolds can be constructed in 2-parameter continuation when both parameters are allowed to vary, see~\cite{GAMEIRO_RigorousMultiparam}.

Fig.~\ref{fig:ContinuationDiagram} shows the norm (a) and offset energy (b) of the main ansatz at $(\pbar, \beta) = (0.07, 0.025)$ are plotted as functions of $\pbar$. The domain is kept small with $(N_x, N_y) = (2, 1)$ to keep the bifurcation diagram as simple as possible. The atoms and donuts branches are actually the same since we can continue the branches through the folds at $\pbar = \pm \sqrt{5/12\beta}$. This branch intersects the checkers state at $\beta = 15\pbar^2$ and the constant and stripes states at $\beta = 3\pbar^2$. The energy plot (b) clearly shows that the donuts state is the ``proper'' hexagonal lattice for $\pbar < 0$. We note that varying $\beta$ simply causes the branches to dilate. For example, we expect the 2D hexagonal steady states manifold to be a ``conic'' figure-eight.

\begin{figure}[H]
	\centering
	\subfloat[\hspace*{-1.1cm}]{\includegraphics[width=0.45\textwidth]{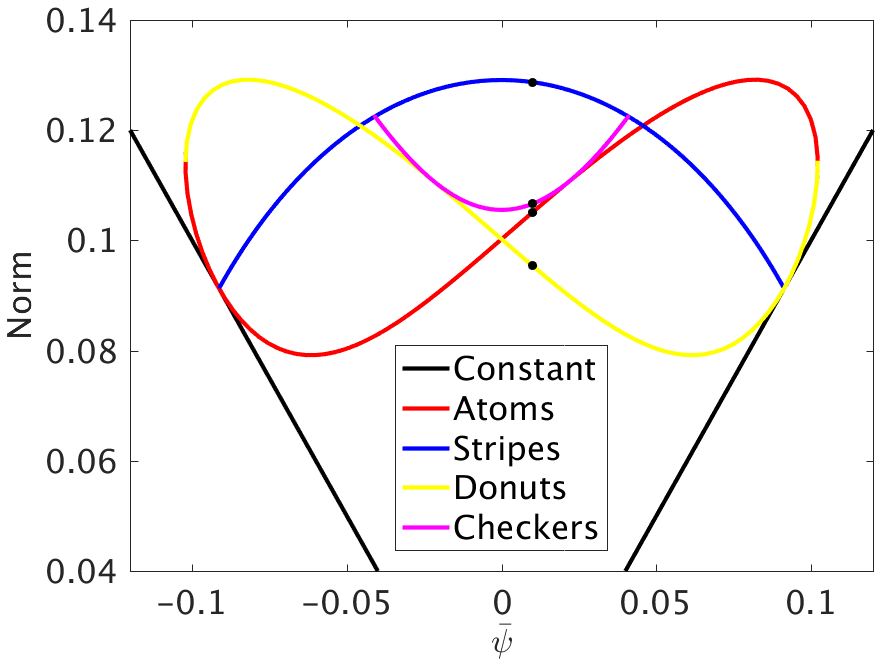}}\enskip
	\subfloat[\hspace*{-0.95cm}]{\includegraphics[width=0.45\textwidth]{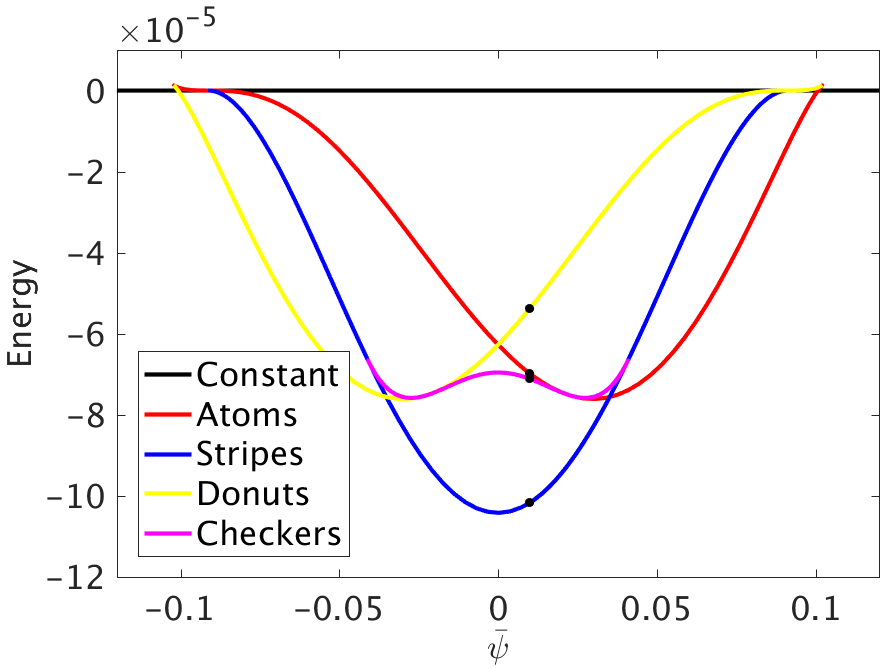}}
	
	\caption[Continuation diagram]{Continuation (bifurcation) diagram showing the $\lp{2}$ norm of the phase (a) and the energy offset by $E_0$ (b) as functions of $\pbar$. The dots represent the starting points at $(\pbar, \beta) = (0.01, 0.025)$.}
	\label{fig:ContinuationDiagram}
\end{figure}

Other ``new'' branches will appear for larger domains or higher $\beta$. In particular, Fig.~\ref{fig:ContinuationDiagramLocalized} shows the atoms/donuts branch and the single atom branch in the localized patterns regime near $(\pbar, \beta) = (0.5, 0.6)$ with $(N_x, N_y) = (7, 4)$. Again, (a) shows the $\lp{2}$ norm and (b) shows the energy of the phase field as functions of $\pbar$. The hexagonal lattice traces out its usual figure-eight pattern while the single atom (and other localized states in general) traces out a complicated looping path. Such branches illustrate the ``snaking'' phenomenon previously observed in modified Swift-Hohenberg equations that support such localized patterns, see~\cite{LLOYD_LocalizedHexagons} for example. We observe that the path loops on itself in one direction as the single atom evolves into a localized pattern with $9$, $7$ then $4$ atoms before looping back with a $90^\circ$ rotation. In the other direction, the branch moves towards the transition between the hexagonal and constant states where it again loops back. This computation is difficult because the truncation must remain large and the pseudo-arclength step size must remain small; if the step size is larger than $0.0005$, the branch breaks away towards the hexagonal lattice solution.
\begin{figure}[H]
	\centering
	\subfloat[\hspace*{-1.1cm}]{\includegraphics[width=0.45\textwidth]{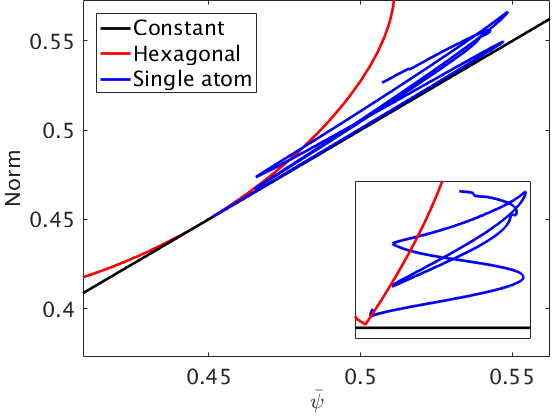}}\enskip
	\subfloat[\hspace*{-0.95cm}]{\includegraphics[width=0.45\textwidth]{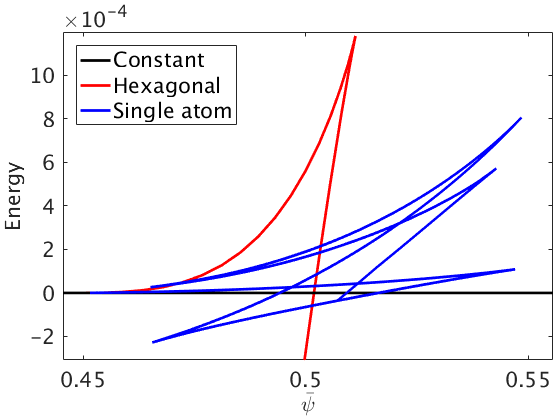}}
	
	\caption[Continuation diagram]{Continuation (bifurcation) diagram showing the $\lp{2}$ norm of the phase (a) and the energy offset by $E_0$ (b) as functions of $\pbar$. The inset in (a) shows the norm of $\psi - \pbar$ to better illustrate the snaking phenomenon. Both the hexagonal and single atom branches appear to loop on themselves.}
	\label{fig:ContinuationDiagramLocalized}
\end{figure}

\section{Conclusion}
We surveyed the basic properties of the PFC equation as a dynamical system in the framework of rigorous numerics. Thanks to an application of the radii polynomial approach, we were able to verify the existence of steady states close to numerically computed approximations. This provided us with important \textit{verified} information as to the behavior of the energy landscape, especially in terms of energetic relationships between steady states. We were also able to provide partial stability results with the caveat that they only applied to the cosine Fourier series. The Morse indices given were lower bounds in $\hp{2}$ - thus those steady states with Morse index higher than 0 must be unstable in $\hp{2}$.

Such ideas were applied in various regimes of the PFC equation to verify that certain important patterns are steady states (as opposed to metastable intermediates) including single atoms, other localized patterns and grain boundaries. In particular, we showed that two-mode PFC supports a non-zero misorientation grain boundary \textit{steady state} that we expect to be stable. We also showed the construction of the phase diagram with our fully nonlinear approach.

Finally, we used such results to further investigate the energy landscape through connections or orbits and through parameter continuation. Connections reveal the energetic and dynamical relationships between steady states, highlighting the behavior of unstable patterns as they reach states with lower energy. Continuation is especially useful to understand how the important states evolve across parameter space, highlighting the surprising behavior of the hexagonal lattice patterns and the snaking behavior of localized patterns. 

Our work suggests several interesting directions for future work. On one hand, our connection results could be made rigorous with a technique to prove orbits from unstable to stable manifolds. This is a complicated problem because the stable manifold is infinite dimensional and special techniques must be applied to properly parameterize its ``dominant'' submanifold. On the other, our continuation results could also be made rigorous or extended to 2-parameter continuation to reveal more interesting behavior. Alternatively, parameter continuation could be applied to the domain size, for example to investigate problems in elasticity.


\bibliographystyle{ieeetr}
\bibliography{SP_References}

\begin{thebibliography}{10}

\bibitem{ELDER_Elasticity}
K.~R. Elder, M.~Katakowski, M.~Haataja, and M.~Grant, ``Modeling elasticity in
  crystal growth,'' {\em Physical Review Letters}, vol.~88, p.~245701, 2002.

\bibitem{SWIFT_Hohenberg}
J.~Swift and P.~C. Hohenberg, ``Hydrodynamic fluctuations at the convective
  instability,'' {\em Physical Review A}, vol.~15, pp.~319--328, 1977.

\bibitem{MARTINE_AtomBased}
G.~Martine La~Boissoni\`ere and R.~Choksi, ``Atom based grain extraction and
  measurement of geometric properties,'' {\em Modelling and Simulation in
  Materials Science and Engineering}, vol.~26, no.~3, p.~035001, 2018.

\bibitem{MARTINE_PFCStatistics}
G.~Martine La~Boissoni\`ere, R.~Choksi, K.~Barmak, and S.~Esedo\={g}lu,
  ``Statistics of grain growth: experiment versus the {P}hase-{F}ield-{C}rystal
  and {M}ullins models,'' {\em Materialia}, p.~100280, 2019.

\bibitem{BACKOFEN_GSD}
R.~Backofen, K.~Barmak, K.~E. Elder, and A.~Voigt, ``Capturing the complex
  physics behind universal grain size distributions in thin metallic films,''
  {\em Acta Materialia}, vol.~64, pp.~72--77, 2014.

\bibitem{EMMERICH_PFCReview}
H.~{Emmerich}, H.~{L{\"o}wen}, R.~{Wittkowski}, T.~{Gruhn}, G.~I. {T{\'o}th},
  G.~{Tegze}, and L.~{Gr{\'a}n{\'a}sy}, ``Phase-field-crystal models for
  condensed matter dynamics on atomic length and diffusive time scales: an
  overview,'' {\em Advances in Physics}, vol.~61, no.~6, pp.~665--743, 2012.

\bibitem{GREENWOOD_Eutectic}
M.~Greenwood, N.~Ofori-Opoku, J.~Rottler, and N.~Provatas, ``Modeling
  structural transformations in binary alloys with phase field crystals,'' {\em
  Physical Review B}, vol.~84, no.~6, p.~064104, 2011.

\bibitem{SEYMOUR_Graphene}
M.~Seymour and N.~Provatas, ``Structural phase field crystal approach for
  modeling graphene and other two-dimensional structures,'' {\em Physical
  Review B}, vol.~93, p.~035447, 2016.

\bibitem{HIRVONEN_Graphene}
P.~Hirvonen, M.~M. Ervasti, Z.~Fan, M.~Jalalvand, M.~Seymour, S.~M.~V. Allaei,
  N.~Provatas, A.~Harju, K.~R. Elder, and T.~Ala-Nissila, ``Multiscale modeling
  of polycrystalline graphene: a comparison of structure and defect energies of
  realistic samples from phase field crystal models,'' {\em Physical Review B},
  vol.~94, no.~3, p.~035414, 2016.

\bibitem{SHIROKOFF_GlobalMinimality}
D.~Shirokoff, R.~Choksi, and J.-C. Nave, ``Sufficient conditions for global
  minimality of metastable states in a class of non-convex functionals: a
  simple approach via quadratic lower bounds,'' {\em Journal of Nonlinear
  Science}, vol.~25, no.~3, pp.~539--582, 2015.

\bibitem{KOCH_ComputeAssisted}
H.~Koch, A.~Schenkel, and P.~Wittwer, ``Computer-assisted proofs in analysis
  and programming in logic: a case study,'' {\em SIAM Review}, vol.~38, no.~4,
  pp.~565--604, 1996.

\bibitem{NAKAO_VerifiedPDE}
M.~T. Nakao, ``Numerical verification methods for solutions of ordinary and
  partial differential equations,'' {\em Numerical Functional Analysis and
  Optimization}, vol.~22, no.~3-4, pp.~321--356, 2001.

\bibitem{TUCKER_ValidatedIntroduction}
W.~Tucker, {\em Validated numerics: a short introduction to rigorous
  computations}.
\newblock Princeton University Press, 2011.

\bibitem{VANDENBERG_Dynamics}
J.~B. van~den Berg and J.~P. Lessard, ``Rigorous numerics in dynamics,'' {\em
  Notices of the American Mathematical Society}, vol.~62, no.~9,
  pp.~1057--1061, 2015.

\bibitem{GOMEZ_PDESurvey}
J.~G{\'o}mez-Serrano, ``Computer-assisted proofs in {PDE}: a survey,'' {\em
  SeMA Journal}, pp.~1--26, 2018.

\bibitem{WU_TwoMode}
K.~A. Wu, A.~Adland, and A.~Karma, ``Phase-field-crystal model for {FCC}
  ordering,'' {\em Physical Review E}, vol.~81, no.~6, p.~061601, 2010.

\bibitem{DAY_ValidatedContinuation}
S.~Day, J.-P. Lessard, and K.~Mischaikow, ``Validated continuation for
  equilibria of {PDE}s,'' {\em SIAM Journal on Numerical Analysis}, vol.~45,
  no.~4, pp.~1398--1424, 2007.

\bibitem{HUNGRIA_RadiiPolynomial}
A.~Hungria, J.~P. Lessard, and J.~D. Mireles~James, ``Rigorous numerics for
  analytic solutions of differential equations: the radii polynomial
  approach,'' {\em Mathematics of Computation}, vol.~85, no.~299,
  pp.~1427--1459, 2016.

\bibitem{BALASZ_RadiallySymmetric}
I.~Bal{\'a}zs, J.~B. van~den Berg, J.~Courtois, J.~Dud{\'a}s, J.~P. Lessard,
  A.~V{\"o}r{\"o}s-Kiss, J.~F. Williams, and X.~Y. Yin, ``Computer-assisted
  proofs for radially symmetric solutions of {PDE}s,'' {\em Journal of
  Computational Dynamics}, vol.~5, no.~1-2, pp.~61--80, 2018.

\bibitem{VANDENBERG_RigorousChaos}
J.~B. van~den Berg, ``Introduction to rigorous numerics in dynamics: general
  functional analytic setup and an example that forces chaos,'' {\em Rigorous
  Numerics in Dynamics, Proceedings of Symposia in Applied Mathematics},
  vol.~74, pp.~1--25, 2017.

\bibitem{VANDENBERG_OhtaKawasaki}
J.~B. van~den Berg and J.~F. Williams, ``Validation of the bifurcation diagram
  in the {2D} {O}hta-{K}awasaki problem,'' {\em Nonlinearity}, vol.~30, no.~4,
  p.~1584, 2017.

\bibitem{VANDENBERG_OhtaKawasaki2}
J.~B. van~den Berg and J.~F. Williams, ``Rigorously computing symmetric
  stationary states of the {O}hta-{K}awasaki problem in three dimensions,''
  {\em SIAM J. Math. Anal.}, vol.~51, no.~1, pp.~131--158, 2017=9.

\bibitem{MOORE_IntervalAnalysis}
R.~E. Moore, {\em Interval analysis}, vol.~4.
\newblock Prentice-Hall, 1966.

\bibitem{RUMP_INTLAB}
S.~M. Rump, ``{INTLAB} - interval laboratory,'' in {\em Developments in
  Reliable Computing}, pp.~77--104, Springer, 1999.

\bibitem{HARGREAVES_IntervalMATLAB}
G.~I. Hargreaves, ``Interval analysis in {MATLAB},'' {\em Numerical
  Algorithms}, no.~2009.1, 2002.

\bibitem{ELSEY_Scheme}
M.~Elsey and B.~Wirth, ``A simple and efficient scheme for phase field crystal
  simulation,'' {\em ESAIM: Mathematical Modelling and Numerical Analysis},
  vol.~47, no.~5, pp.~1413--1432, 2013.

\bibitem{LLOYD_LocalizedHexagons}
D.~J.~B. Lloyd, B.~Sandstede, D.~Avitabile, and A.~R. Champneys, ``Localized
  hexagon patterns of the planar {S}wift-{H}ohenberg equation,'' {\em SIAM
  Journal on Applied Dynamical Systems}, vol.~7, no.~3, pp.~1049--1100, 2008.

\bibitem{VANDENBERG_CoexistenceHexagonsRolls}
J.~B. van~den Berg, A.~Desch{\^e}nes, J.~P. Lessard, and J.~D. Mireles~James,
  ``Stationary coexistence of hexagons and rolls via rigorous computations,''
  {\em SIAM Journal on Applied Dynamical Systems}, vol.~14, no.~2,
  pp.~942--979, 2015.

\bibitem{BERG_JAQUETTE_MIRELES}
J.~van~den Berg, J.~Jaquette, and J.~Mireles~James, ``Validated numerical
  approximation of stable manifolds for parabolic partial differential
  equations.'' Preprint, 2020.

\bibitem{MR3773757}
J.~Cyranka and T.~Wanner, ``Computer-assisted proof of heteroclinic connections
  in the one-dimensional {O}hta-{K}awasaki {M}odel,'' {\em SIAM J. Appl. Dyn.
  Syst.}, vol.~17, no.~1, pp.~694--731, 2018.

\bibitem{MR3906120}
C.~Reinhardt and J.~D. Mireles~James, ``Fourier-{T}aylor parameterization of
  unstable manifolds for parabolic partial differential equations: formalism,
  implementation and rigorous validation,'' {\em Indag. Math. (N.S.)}, vol.~30,
  no.~1, pp.~39--80, 2019.

\bibitem{JAQUETTE_LESSARD_TAKAYASU}
J.~Jaquette, J.-P. Lessard, and A.~Takayasu, ``Global dynamics in
  nonconservative nonlinear {S}chr\"odinger equations.'' Preprint, 2020.

\bibitem{VANDENBERG_BranchFollowing}
J.~B. van~den Berg, J.~P. Lessard, and K.~Mischaikow, ``Global smooth solution
  curves using rigorous branch following,'' {\em Mathematics of Computation},
  vol.~79, no.~271, pp.~1565--1584, 2010.

\bibitem{MR910499}
H.~B. Keller, {\em Lectures on numerical methods in bifurcation problems},
  vol.~79 of {\em Tata Institute of Fundamental Research Lectures on
  Mathematics and Physics}.
\newblock Published for the Tata Institute of Fundamental Research, Bombay,
  1987.
\newblock With notes by A. K. Nandakumaran and Mythily Ramaswamy.

\bibitem{GAMEIRO_RigorousMultiparam}
M.~Gameiro, J.~P. Lessard, and A.~Pugliese, ``Computation of smooth manifolds
  via rigorous multi-parameter continuation in infinite dimensions,'' {\em
  Foundations of Computational Mathematics}, vol.~16, no.~2, pp.~531--575,
  2016.

\end{thebibliography}

\newpage
\appendix
\section*{Appendices}
\addcontentsline{toc}{section}{Appendices}
\renewcommand{\thesubsection}{\Alph{subsection}}

\subsection{PFC ansatz in 2D}
\label{sm:PFCAnsatz}
PFC simulations can be classified according to a small number of regimes or ansatz that represent the (expected) global minimizer. The choice of such candidates is motivated by numerical experiments but can be obtained analytically from techniques such as linear stability analysis. Consider a periodic ``single Fourier mode'' phase field of the form
\begin{equation*}
\psi(x, y) = \pbar + A_1 \cos(y) + A_2 \cos \left( \frac{\sqrt{3}}{2} x - \frac{1}{2} y \right) + A_3 \cos \left( \frac{\sqrt{3}}{2} x + \frac{1}{2} y \right)
\end{equation*}
on the rectangular domain $[0, 4\pi/\sqrt{3}] \times [0, 4\pi]$. Inserting this ansatz into the PFC energy yields an expression $E[A_1, A_2, A_3]$ that can be optimized in the three amplitudes. This procedure yields three main classes of states that are well-known in the PFC literature.

\begin{itemize}
\item The constant state $A_1 = A_2 = A_3 = 0$.
\item The stripes state $A_2 = A_3 = 0$ and $A_1 = \frac{2}{\sqrt{3}} \sqrt{\beta - 3\pbar^2}$.
\item The hexagonal lattice states $A_1 = A_2 = A_3 = \frac{-2\pbar}{5} \pm \frac{2}{\sqrt{15}} \sqrt{\beta - \frac{12}{5} \pbar^2}$.
\end{itemize}
In addition, we also find a ``checkers state'' where $A_1$ and $A_2 = A_3$ are given as more complicated expressions. The two hexagonal lattices differ in energy: the positive choice is called the ``donuts'' state while the negative one is the ``atoms'' state. We can compare the energies directly to show that the checkers state is never optimal while the atoms state is more optimal than the donuts state for $\pbar > 0$. 

When the coefficients of two candidates are equal, they represent the same regime; for example, at $\beta = 3\pbar^2$, the constant, stripes and donuts states are all $\psi(x, y) = \pbar$. Similarly, the donuts and atoms states merge at $\beta = 12/5 \pbar^2$. We can also compute when states have the same energy. Such behavior occurs on transition curves of the form $\beta = \alpha \pbar^2$.

We can construct the phase diagram in Fig.~\ref{fig:AnsatzPhaseDiagram} by labeling with the expected global minimizer. We show in the main text that this ``linear'' description of PFC is a good approximation, at least for small $\beta$.

\begin{figure}
	\centering
	\subfloat[\hspace*{-1cm}]{\includegraphics[width=0.48\textwidth]{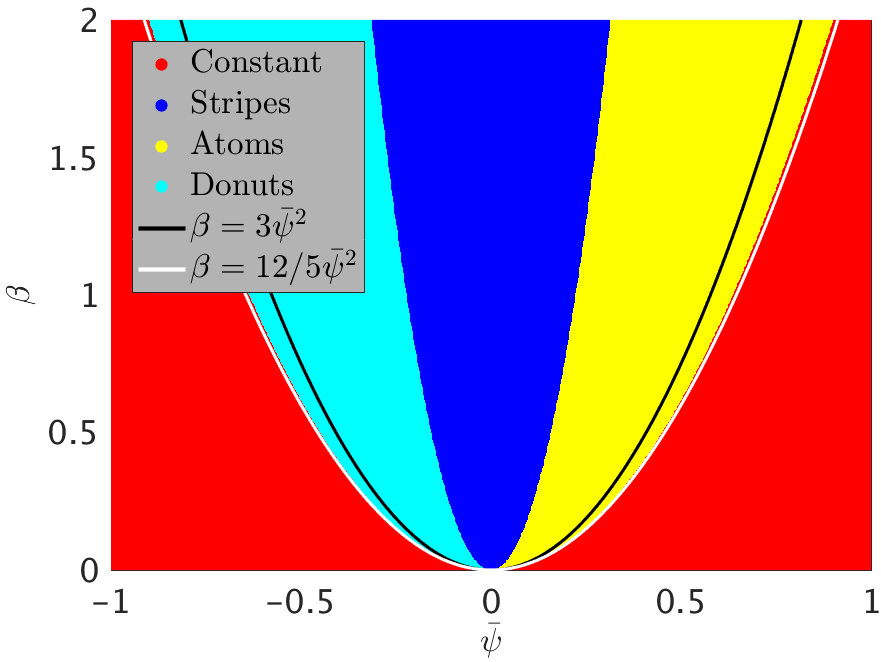}}\enskip
	\subfloat[\hspace*{-1.05cm}]{\includegraphics[width=0.48\textwidth]{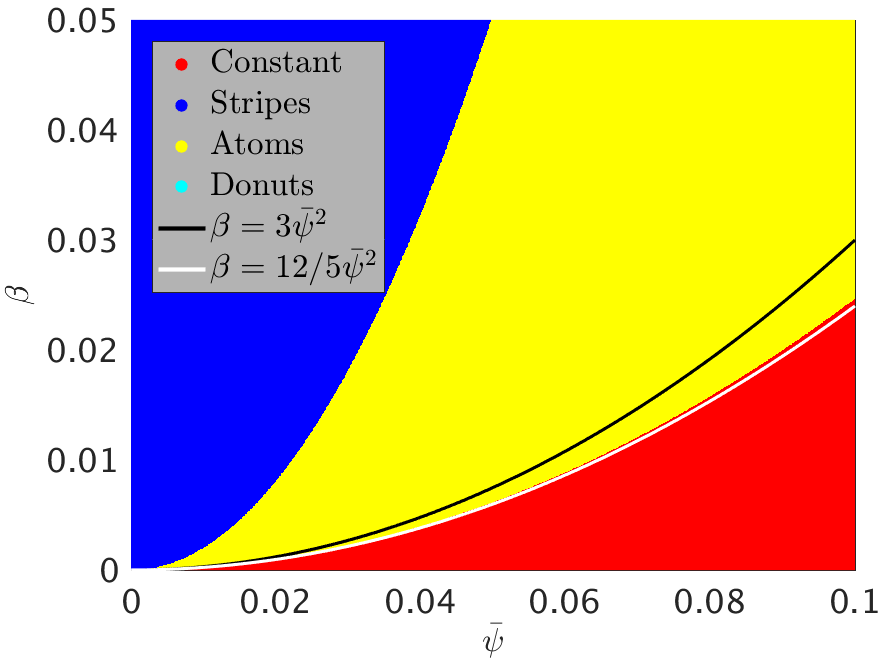}}
	
	\caption[Phase diagram from the basic ansatz]{Phase diagram (a) and detail (b) constructed by comparing the optimal energy of the four ansatz. The curves $\beta = 3\pbar^2 \midand \beta = 12/5 \pbar^2$ correspond respectively to the curves on which the amplitude of the stripes and lattice states become complex. The stripes-atoms (blue-yellow) transition curve is $\beta \approx 20.22 \pbar^2$ while the constant-atoms (red-yellow) transition curve is $\beta = 37/15 \pbar^2$. Donuts are never optimal for $\pbar > 0$. Checkers are never optimal.}
	\label{fig:AnsatzPhaseDiagram}
\end{figure}

\subsection{Proof of the radii polynomial theorem}
\label{sm:RadiiPolyProof}
\begin{proof}
Consider the Newton operator $T(a) = a - AF(a)$, then $T : X \to X$ and any fixed point $\atil$ of $T$ is a zero of $F$ because $A$ is injective. The derivative of $T$ is bounded and also Fr\'echet differentiable since we have for any $x \in X$
\begin{equation*}
\normbX{DT(x)} = \normbX{I - A DF(x)} \leq \normbX{I} + \normbX{A DF(x)} < \infinity
\end{equation*}
as $DF$ is the (bounded linear) Fr\'echet derivative of $F$. Now suppose $p(r_0) < 0$ for some $r_0 > 0$, then the radii polynomial in the main text gives
\begin{equation*}
Z_2(r_0)r_0 + Z_0 + Z_1 < 1 + p(r_0)/r_0 - Y_0/r_0 < 1
\end{equation*}
since $Y_0$ is positive.

Let $a \in \overline{B_{r_0}(\abar)}$, we can use the previous inequality to bound
\begin{equation*}
\begin{split}
\normbX{&DT(a)} = \normbX{I - AA^\dagger + AA^\dagger - ADF(\abar) + ADF(\abar) - ADF(a)} \\
	&\leq \normbX{I - AA^\dagger} + \normbX{A(A^\dagger - DF(\abar))} + \normbX{A(DF(\abar) - DF(a))} \\
	&\leq Z_0 + Z_1 + Z_2(r_0) r_0 \eqdot
\end{split}
\end{equation*}

Pairing this with the mean value inequality for $T$,
\begin{equation*}
\begin{split}
\normX{T(a) - \abar} &= \normX{T(a) - T(\abar) + T(\abar) - \abar} \\
	&\leq \sup_{z \in \overline{B_{r_0}(\abar)}} \normbX{DT(z)} \normX{a-\abar} + \normX{AF(\abar)} \\
	&\leq \left( Z_0 + Z_1 + Z_2(r_0) r_0 \right) r_0 + Y_0 \\
	&= p(r_0) + r_0 < r_0
\end{split}
\end{equation*}
hence $T$ maps $\overline{B_{r_0}(\abar)}$ to its \textit{interior} thanks to the strict inequality. Similarly for $x, y \in \overline{B_{r_0}(\abar)}$,
\begin{equation*}
\begin{split}
\normX{T(x) - T(y)} &\leq \sup_{z \in \overline{B_{r_0}(\abar)}} \normbX{DT(z)} \normX{x-y} \\
	&\leq (Z_0 + Z_1 + Z_2(r_0) r_0) \normX{x-y} < \normX{x-y}
\end{split}
\end{equation*}
so $T: \overline{B_{r_0}(\abar)} \to B_{r_0}(\abar)$ is a contraction with constant $\kappa = (Z_0 + Z_1 + Z_2(r_0) r_0) < 1$ and the Banach fixed-point theorem gives the result.
\end{proof}

\subsection{Computation of the radii polynomial bounds}
\label{sm:SimplifiedBounds}
In the following calculations, we will use usual results such as
\begin{equation*}
\normnu{Qb} \leq \normb{Q} \normnu{b} \eqcom \quad \normb{Q} = \sup_{\normnu{b} = 1} \normnu{Qb} \eqcom
\end{equation*}
and the following proposition to compute the norm of operators on $X$:\\
\begin{proposition}
\label{prop:QNorm}
\textit{Let $Q$ be an operator such that $Q_{\alpha, \sigma} = c_\alpha \delta_{\sigma_1-\alpha_1} \delta_{\sigma_2-\alpha_2}$ whenever $\alpha, \sigma \notin U = \{0, 1, ..., M\}^2$, then}
\begin{equation*}
\normb{Q} \leq \max_{\alpha \in U} \left\{ \frac{1}{\nu_\alpha} \sum_{\sigma \in U} |Q_{\alpha, \sigma}| \nu_\sigma \right\} + \sup_{\alpha \notin U} |c_\alpha| \eqdot
\end{equation*}
\end{proposition}

\begin{proof}
Let $b \in X$, then $Qb$ can be decomposed as the action of the first finite block onto $b_\sigma$ for $\sigma \in U$ and infinitely many diagonal terms $c_\sigma b_\sigma$ for $\sigma \notin U$. The norm of $Q$ can then be written as the sum of two disjoint positive sums using the triangle inequality:
\begin{equation*}
\begin{split}
\normnu{Qb} &= \sum_{\sigma \in N^2} |(Qb)_\sigma| \nu_\sigma = \sum_{\sigma \in U} \left| \sum_{\alpha \in U} Q_{\alpha, \sigma} b_\alpha \right| \nu_\sigma + \sum_{\sigma \notin U} |c_\sigma b_\sigma| \nu_\sigma \\
	&\leq \sum_{\sigma \in U} \sum_{\alpha \in U} |Q_{\alpha, \sigma}| |b_\alpha| \nu_\sigma + \sum_{\sigma \notin U} |c_\sigma| |b_\sigma| \nu_\sigma \\
	&\leq \sum_{\alpha \in U} \left( \frac{1}{\nu_\alpha} \sum_{\sigma \in U} |Q_{\alpha, \sigma}| \nu_\sigma \right) |b_\alpha| \nu_\alpha + \sum_{\alpha \notin U} |c_\sigma| |b_\alpha| \nu_\alpha
\end{split}
\end{equation*}

The second term is bounded by the trivial bound $C = \sup_{\alpha \notin U} |c_\alpha|$, extracting the norm of $b$ over $\N \setminus U$. Similarly, the first term is bounded by
\begin{equation*}
\max_{\alpha \in U} \left\{ \frac{1}{\nu_\alpha} \sum_{\sigma \in U} |Q_{\alpha, \sigma}| \nu_\sigma \right\} \sum_{\alpha \in U} |b_\alpha| \nu_\alpha = K \sum_{\alpha \in U} |b_\alpha| \nu_\alpha \leq K \normnu{b} \eqdot
\end{equation*}

The norm of $Q$ in $B(\ellnu)$ is the supremum of the previous norm over all $b$ with unit norm, therefore the triangle inequality gives the result.
\end{proof}

The sharper result $\max\{K, C\}$ can be obtained by noting that the sums act on different subspaces of $\ellnu$. The estimate also allows us to compute the norm of \textit{finite} tensors by letting the $c_\alpha$ vanish. Let us now apply these results to compute the four necessary bounds.

\subsubsection{$Z_0$ bound}
This bound is the easiest since by construction, $A \midand A^\dagger$ are approximate inverses up to numerical inversion errors. We then have
\begin{equation*}
Z_0 = \normb{I - AA^\dagger} = \normb{I^{(M)} - A^{(M)} G}
\end{equation*}
which can be evaluated using proposition~\ref{prop:QNorm}.

\subsubsection{$Y_0$ bound}
We must compute $AF(\abar)$ so that most terms will be given by the finite product $A^{(M)} F^{(M)}(\abar)$. There still remain some non-zero convolution coefficients in the full $F(\abar)$: since $\abar^{(M)}$ has $M+1$ coefficients in each dimension, $L_\alpha (\abar*\abar*\abar)_\alpha$ will have $3M+1$ non-zero coefficients in each dimension. These are multiplied by the appropriate $L^{-1}_\alpha \gamma^{-1}_\alpha$, resulting in
\begin{equation*}
\normnu{AF(\abar)} \leq \normnu{A^{(M)} F^{(M)}(\abar)} + \sum_{\alpha \in \{0, 1, ..., 3M\}^2 \setminus U} \left| \frac{(\abar*\abar*\abar)_\alpha}{\gamma_\alpha} \right| \nu_\alpha \eqdot
\end{equation*}

\subsubsection{$Z_2$ bound}
To compute the $Z_2$ bound, let $b, h \in X$ with $\normnu{h} = 1$ and consider first the effect of $DF(b)$ on $h$,
\begin{equation*}
(DF(b) h)_\alpha = \left. \frac{d}{ds} F_\alpha(b + s h) \right|_{s = 0} = L_\alpha \left(\gamma_\alpha h_\alpha + 3(b * b * h)_\alpha \right) \eqdot
\end{equation*}

Fix $r > 0$ and let $b = \abar + R$ where $\normnu{R} \leq r$, we then have that
\begin{equation*}
((DF(b)-DF(\abar)) h)_\alpha = 3 L_\alpha (b * b * h - \abar * \abar * h)_\alpha = 3 L_\alpha ((2\abar * R + R * R) * h)_\alpha \eqdot
\end{equation*}

Note that the initial factor $L_\alpha$ can instead be represented by the diagonal operator defined by $\Lambda_{\alpha, \sigma} = L_\alpha \delta_{\alpha_1-\sigma_1} \delta_{\alpha_2-\sigma_2}$. Then, using the fact that the convolution on $X$ is a Banach algebra,
\begin{equation*}
\begin{split}
\normb{A(DF(b)-DF(\abar))} &\leq 3 \normb{A\Lambda} \normnu{2\abar * R + R * R} \normnu{h} \\
	&\leq 3 \normb{A\Lambda} (2\normnu{\abar} + r) r
\end{split}
\end{equation*}
where the norm of $A\Lambda$ is computed using proposition~\ref{prop:QNorm} with the bound $\Gamma = \max(\gamma^{-1}_{0, M+1}, \gamma^{-1}_{M+1, 0})$ on the diagonal terms.\footnote{This can be done since $\gamma_\alpha$ is monotone decreasing as long as \textit{either} $\alpha_1$ or $\alpha_2$ can dominate the $-\beta$ term. For PFC, $\gamma_\alpha = (L_\alpha+1)^2-\beta$ so a sufficiently large $M$ ensures that $M/L_x$ or $M/L_y$ dominates $\beta$. This condition can be checked numerically.} This computation works for any $r > 0$, so we have
\begin{equation*}
Z_2(r) = 6\normb{A\Lambda} \normnu{\abar} + 3 \normb{A\Lambda} r = Z_2^{(0)} + Z_2^{(1)} r \eqdot
\end{equation*}

\subsubsection{$Z_1$ bound}
For the final bound, we now consider the action of $A^\dagger$ on the same vector $h$:
\begin{equation*}
(A^\dagger h)_\alpha = \begin{cases}
	h_\alpha \quad& \text{if } \alpha = (0, 0) \\
	\sum_\sigma G_{\alpha, \sigma} h_\sigma \quad& \text{if } \alpha \in U \setminus \{(0,0)\} \\
	L_\alpha \gamma_\alpha h_\alpha \quad& \text{otherwise}
\end{cases}
\end{equation*}

Let $\eta$ be the tail of $h$, i.e. the vector with the same entries as $h$ outside of $U$ and $0$ on $U$. We then have:
\begin{equation*}
((DF(\abar) - A^\dagger) h)_\alpha = \begin{cases}
	0 \quad& \text{if } \alpha = (0, 0) \\
	3 L_\alpha (\abar*\abar*\eta)_\alpha \quad& \text{if } \alpha \in U \setminus \{(0,0)\} \\
	3 L_\alpha (\abar*\abar*h)_\alpha \quad& \text{otherwise}
\end{cases}
\end{equation*}

Consider now the action of $A$ on the difference above. The first block of the difference will be multiplied by the inverse of $G$ while the tail will be multiplied by the appropriate $L^{-1}_\alpha \gamma^{-1}_\alpha$, thus
\begin{equation*}
\normnu{A(DF(\abar) - A^\dagger) h} \leq 3 \sum_{\sigma \in U} \sum_{\alpha \in U} \left| A_{\alpha, \sigma} L_\alpha (\abar*\abar*\eta)_\alpha \right| \nu_\sigma + 3 \sum_{\sigma \in \Z^2 \backslash U} \left| \frac{(\abar*\abar*h)_\sigma}{\gamma_\sigma} \right| \nu_\sigma
\end{equation*}
using $L_{0,0} = 0$ and the triangle inequality for the first term.

Let $\phi \in X$ be such that $|(\abar * \abar * \eta)_\alpha| \leq \phi_\alpha$ whenever $\alpha \in U$ and $0$ otherwise. Using the Banach algebra property and the bound $\Gamma \geq |\gamma_\sigma|^{-1}$ to bound the infinite sum, we have
\begin{equation*}
\begin{split}
\normb{A(DF(\abar) - A^\dagger)} &\leq 3 \sum_{\sigma \in U} |(A\Lambda \phi)_\sigma| \nu_\sigma + 3 \Gamma \normnu{\abar}^2\\
&= 3 \normnu{A\Lambda \phi} + 3 \Gamma \normnu{\abar}^2 = Z_1
\end{split}
\end{equation*}
which can be computed numerically once the (finitely many) $\phi_\alpha$ have been obtained. To compute them, we now shift for a moment to $\Z^2$ and extend all vectors appropriately. Now let $q = \abar*\abar$, then
\begin{equation*}
\begin{split}
|(q * \eta)_\alpha| &\leq \left| \sum_{\sigma \in \Z^2} q_{\alpha-\sigma} \eta_\sigma \right| \leq \sum_{\sigma \in V(\alpha) \cap W} \left(\frac{|q_{\alpha-\sigma}|}{\nu^{|\sigma|}} \right) |h_\sigma| \nu^{|\sigma|} \\
	&\leq \sum_{\sigma \in V(\alpha) \cap W} \left(\sup_{\tau \in V(\alpha) \cap W} \frac{|q_{\alpha-\tau}|}{\nu^{|\tau|}} \right) |h_\sigma| \nu^{|\sigma|} \\
	&\leq \sup_{\sigma \in V(\alpha) \cap W} \frac{|q_{\alpha-\sigma}|}{\nu^{|\sigma|}} \normnu{h}
\end{split}
\end{equation*}
where $V(\alpha), W \subset \Z^2$ are the regions over which $q_{\alpha-\sigma}$ and $\eta$ are non-zero respectively. Since $q_\tau = 0$ whenever either $|\tau_1|$ or $|\tau_2|$ is larger than $2M$ and we only need $|\alpha| \in U$, we obtain the overestimate that $V(\alpha) \subset \{-3M, ..., 3M\}^2$. Further, $\eta_\sigma$ must vanish for $|\sigma| \in U$ so
\begin{equation*}
\phi_\alpha = \max_{\sigma \in \{-3M, ..., 3M\}^2 \setminus \{-M, ..., M\}^2} \frac{|(\abar*\abar)_{|\alpha_1-\sigma_1|, |\alpha_2-\sigma_2|}|}{\nu^{|\sigma|}}
\end{equation*}
for $\alpha \in U$ (and $0$ otherwise), which completes the computation of $Z_1$.

\subsection{Energy computation and real space norms}
\label{sm:RigorousEnergy}
The notion of closeness between $\atil$ and $\abar$ extends to their energies. For simplicity, we only handle the basic PFC energy
\begin{equation*}
E[\psi] = \dashint_\Omega \frac{1}{2} (\lapl \psi + \psi)^2 + \frac{1}{4}(\psi^2-\beta)^2 \eqdot
\end{equation*}

We will write without relabeling $E[\psi_{(a)}] = E[a]$ when $\psi_{(a)}$ is the phase field corresponding to the Fourier coefficients $a$. Taking the average of a Fourier series returns its constant mode such that
\begin{equation*}
\begin{split}
E[a] = \left( \frac{1}{2} \left(La+a \right) * \left(La+a \right) + \frac{1}{4} \left(a*a - \beta \delta_{\alpha_1} \delta_{\alpha_2} \right) * \left(a*a - \beta \delta_{\alpha_1} \delta_{\alpha_2} \right) \right)_{0,0} \eqdot
\end{split}
\end{equation*}

In the context of the radii polynomial approach, let $t = \atil - \abar$ such that
\begin{equation*}
\begin{split}
E[\atil] &- E[\abar] = E[\abar + t] - E[\abar] \\
	&= \left( \frac{1}{2} \left( 2t + Lt \right) * (Lt) + \left( \abar + L\abar \right) * (Lt) + \left( (1 - \beta) \abar + \abar*\abar*\abar + L\abar \right) * t \right. \\
	&\enskip \left. + \frac{1-\beta}{2} t*t + \frac{3}{2} \abar*\abar * t*t + \abar * t*t*t + \frac{1}{4}t*t*t*t \right)_{0,0} \eqdot
\end{split}
\end{equation*}

We can simplify $\left( \abar + L\abar \right) * (Lt) = \left( L\abar + L^2\abar \right) * t$ by integrating by parts and using the periodic boundary conditions; for example,
\begin{equation*}
\left( a*(Lb) \right)_{0,0} = \dashint_\Omega \psi_{(a)} \lapl \psi_{(b)} = \dashint_\Omega \lapl \psi_{(a)} \psi_{(b)} = \left( (La)*b \right)_{0,0} \eqdot
\end{equation*}

We now use the fact that $\normnu{t} < r_*$. In addition, we can overestimate $|a_{0,0}| \leq \normnu{a}$ and use the Banach algebra property to obtain the following bound:
\begin{equation*}
\begin{split}
\left| E[\atil] - E[\abar] \right| &\leq \frac{1}{2} \left| \left( \left( 2t + L t \right) * (L t) \right)_{0,0} \right| \\
	&+ \left( |1 - \beta| \enskip \normnu{\abar} + \normnu{\abar}^3 + 2 \normnu{L \abar} + \normnu{L^2 \abar} \right) r_* \\
	&+ \frac{1}{2} \left( |1-\beta| + 3 \normnu{\abar}^2\right) r_*^2 + \normnu{\abar} r_*^3 + \frac{1}{4} r_*^4
\end{split}
\end{equation*}

The term strictly in $t$ has been left as a convolution because it is necessary to control the growth of $L_\alpha$ with the $r_*$ bound directly. To do so, we have
\begin{equation*}
|S| = \left| \sum_{\alpha \in \Z^2} (2t_\alpha + L_\alpha t_\alpha) (L_{-\alpha} t_{-\alpha}) \right| = \left| \sum_{\alpha \in \N^2} W_\alpha (2L_\alpha + L_\alpha^2) t_\alpha^2 \right|
\end{equation*}
which is another way to obtain the previous integration by parts result. Now, the $\ellnu$ norm of $t$ is bounded by $r_*$; each member of the sum satisfies the inequality $W_\alpha |t_\alpha| \nu^{|\alpha|} < r_*$. Overestimating $W_\alpha \geq 1$, we can write
\begin{equation*}
|S| < \sum_{\alpha \in \N^2} W_\alpha (2|L_\alpha| + L_\alpha^2) \left( \frac{r_*}{W_\alpha \nu^{|\alpha|}} \right)^2 < r_*^2 \sum_{\alpha \in \N^2} (2|L_\alpha| + L_\alpha^2) \rho^{|\alpha|}
\end{equation*}
where $\rho = 1/\nu^2 < 1$. To evaluate this sum, we can compute the polynomial geometric series
\begin{equation*}
\sum_{j=0}^\infinity \rho^j = \frac{1}{1-\rho} \eqcom \quad
\sum_{j=0}^\infinity j^2 \rho^j = \frac{\rho^2 + \rho}{(1-\rho)^3} \eqcom \quad
\sum_{j=0}^\infinity j^4 \rho^j = \frac{\rho^4 + 11\rho^3 + 11\rho^2 + \rho}{(1-\rho)^5}
\end{equation*}
which all converge for $\rho < 1$. Note that the sums can be evaluated by differentiating $\sum \rho^{jx}$ with respect to $x = 1$. The terms in $L_\alpha$ can then be expanded and written in such a fashion and assuming that $|S|$ is finite, the sums can be split and separated. We have
\begin{equation*}
\begin{split}
S_1 = \sum_{\alpha \in \N^2} |L_\alpha| \rho^{|\alpha|} &= \sum_{j,k \in \N} \left(\left(\frac{2\pi}{L_x} \right)^2 j^2 + \left(\frac{2\pi}{L_y}\right)^2 k^2 \right) \rho^j \rho^k \\
	&= \left(\left(\frac{2\pi}{L_x} \right)^2 + \left(\frac{2\pi}{L_y}\right)^2 \right) \sum_{j,k \in \N} j^2 \rho^j \rho^k \\
	&= |L_{1,1}| \left(\sum_{j \in \N} j^2 \rho^j \right) \left(\sum_{k \in \N} \rho^k \right) \\
	&= |L_{1,1}| \frac{\rho^2 + \rho}{(1-\rho)^3} \cdot \frac{1}{1-\rho} = |L_{1,1}| \frac{\rho^2 + \rho}{(1-\rho)^4}
\end{split}
\end{equation*}
and similarly,
\begin{equation}
\label{eq:S2Computation}
\begin{split}
S_2 &= \sum_{\alpha \in \N^2} L_\alpha^2 \rho^\alpha \\
	&= (L_{1,0}^2 + L_{0,1}^2) \frac{\rho^4 + 11\rho^3 + 11\rho^2 + \rho}{(1-\rho)^6} + 2 |L_{1,0}L_{0,1}| \frac{\rho^4 + 2\rho^3 + \rho^2}{(1-\rho)^6} \eqdot
\end{split}
\end{equation}

Putting everything together, we arrive at the bound
\begin{equation*}
\begin{split}
|E[\atil] - E[\abar]| &\leq \left( |1 - \beta| \enskip \normnu{\abar} + \normnu{\abar}^3 + 2 \normnu{L \abar} + \normnu{L^2 \abar} \right) r_* \\
	&+ \frac{1}{2} \left(2S_1 + S_2 + |1-\beta| + 3 \normnu{\abar}^2 \right) r_*^2 + \normnu{\abar}r_*^3 + \frac{1}{4} r_*^4
\end{split}
\end{equation*}
which can now be computed numerically. This bound depends strongly on $\nu$ because of its influence on $r_*$ and the growth of $S_1 \midand S_2$. In principle, one could find an optimal $\nu$ that ensures the bound is as small as possible for a given $\abar$ and a fixed $\beta$.

When the numerical errors associated to the $E[\abar]$ computations are added to the energy bound, both computed with interval arithmetic, we obtain an interval that is guaranteed to contain the energy of $\atil$ itself. In particular, this allows us to prove which of two steady states is more optimal\footnote{The energy intervals must be disjoint for such statements to hold. Otherwise, the proof parameters must be improved to tighten the intervals.} strictly from numerical computations.

The previous computations illustrate some techniques that allow us to estimate the norm of $\psi_{(t)}$ (the phase field corresponding to $\atil - \abar$) in terms of $r_*$. For example,
\begin{equation}
\label{eq:SupNormBoundRadius}
\begin{split}
\norm{\psi_{(t)}}_\infinity &= \sup_{\rmbf{x} \in \Omega} |\psi_{(t)}(\rmbf{x})| = \sup_{\rmbf{x} \in \Omega} \left| \sum_{\alpha \in \N^2} W_\alpha t_\alpha \cos \left( \frac{2\pi \alpha_1}{L_x} x \right) \cos \left( \frac{2\pi \alpha_2}{L_y} y \right) \right| \\
	&\leq \sum_{\alpha \in \N^2} W_\alpha |t_\alpha| \leq \sum_{\alpha \in \N^2} W_\alpha |t_\alpha| \nu^{|\alpha|} = \normnu{t} < r_*
\end{split}
\end{equation}
provides a pointwise estimate on the value of the exact steady state. Further, a simple calculation shows that Parseval's identity holds on $X$; i.e.\ $\norm{\psi_{(a)}}_{\lp{2}(\Omega)}^2 = |\Omega| \sum_{\alpha \in \N} W_\alpha a_\alpha^2$. We can then bound the $\lp{2}$ norm of derivatives, for instance
\begin{equation*}
\normp{\lapl \psi_{(t)}}{2}^2 = |\Omega| \sum_{\alpha \in \N} L_\alpha^2 t_\alpha^2 \leq S_2 |\Omega| r_*^2
\end{equation*}
using Eq.~(\ref{eq:S2Computation}). Similar results can be built for the lower norms, thus providing an estimate of the form
\begin{equation*}
\normhp{\psi_{(t)}}{2} \leq C(\Omega, \nu) r_*
\end{equation*}
for some constant $C$ that could be computed if necessary. Note the implicit dependence on the state itself and $(\pbar, \beta)$ through $r_*$. Combined with the $\lp{\infinity}$ bound, this shows that as long as $\nu > 1$, the exact steady state will be in $\hp{2}$ and can differ from the numerical candidate by at most $r_*$ at any point in $\Omega$. The constant $C$ may be large, but it does not affect the pointwise agreement; this is sufficient control for our numerical investigation.

\subsection{Stability in $X$}
\label{sm:RigorousStability}
To complete the analysis of a given steady state $\atil$, we can characterize its stability in $X$. This is powerful because even linear results are mostly limited to trivial states but a major limitation is that this does not transfer to $\hp{2}$ because $X$ is restricted to the \textit{cosine} series.

Suppose we have a steady state $\atil \in B_r(\abar)$ for a verified radius $r$, then stability is controlled by the spectrum of $DF(\atil)$. This spectrum is real because we are in the context of a gradient flow; this can be seen directly from the definition of $A^\dagger \midand DF$ which are symmetric on interchanging indices. Assuming there are no zero eigenvalues, the positive and negative ones define the \textit{unstable} and \textit{stable manifolds} respectively. A steady state with only strictly negative eigenvalues is said to be stable.

While only the approximation $A^\dagger$ is known in practice, it has the same \textit{signature} as $DF(\atil)$ itself; i.e.\ they have exactly as many strictly positive or strictly negative eigenvalues. We compute the spectrum of $A^\dagger$ in two parts. The $(M+1)^2$ eigenvalues of the finite block $G = DF^{(M)}(\abar)$ can be computed numerically and verified using interval arithmetic routines.\footnote{This verification is numerically costly for large truncation order; when $M > 40$, we only verify the eigenvalues that are larger than some arbitrary lower bound, say $-1$} Most eigenvalues can be unequivocally assigned a sign, but some may be identically $0$ or closer to $0$ than the available precision. Stability cannot be ascertained in such cases, but assuming $\abar$ was verified using the radii polynomial approach, $G$ must be sufficiently well-conditioned so its eigenvalues cannot be so small.

In the tail, the eigenvalues are simply equal to the diagonal terms $L_\alpha \gamma_\alpha$ with $\alpha \notin U$. Thankfully, $L$ is strictly negative for $\alpha \neq (0,0)$ and $\gamma$ is strictly positive as long as $M$ is sufficiently large. We then have
\begin{equation*}
\sigma(A^\dagger) = \sigma(G) \cup \left\{L_\alpha \gamma_\alpha \right\}_{\alpha \in \N^2 \setminus U}
\end{equation*}
which can be split into a \textit{finite} number of positive eigenvalues and infinitely many negative eigenvalues, assuming there are no small eigenvalues. The equivalence with $\sigma(DF(\atil))$ follows from a homotopy argument. Let $H_s = (1-s) A^\dagger + s DF(\atil)$ for $s \in [0,1]$, then
\begin{equation*}
\begin{split}
\normb{I &- AH_s} = \normb{I - AA^\dagger - s A (DF(\abar) - A^\dagger) + s A (DF(\abar) - DF(\atil))} \\
	&\leq \normb{I - AA^\dagger} + s \normb{A(DF(\abar) - A^\dagger)} + s \normb{A(DF(\abar)-DF(\atil))}\\
	&\leq Z_0 + s Z_1 + s Z_2(r) r \leq Z_0 + Z_1 + Z_2(r) r < 1
\end{split}
\end{equation*}
as in the proof of the radii polynomial approach. Since $I-AH_s$ is a bounded operator with norm less than $1$, $AH_s$ is itself invertible. $AH_s$ and thus $H_s$ cannot have a zero eigenvalue so that its signature must stay constant for all $s$. This shows that $A^\dagger \midand DF(\atil)$ have the same signature and this is in fact true over the ball of radius $r^*$ around $\abar$.

\end{document}